\documentclass[reqno]{amsart}
\usepackage{graphicx}
\usepackage{amsmath,amsfonts,amssymb,hyperref,mathrsfs,esint}
\DeclareMathOperator{\supp}{supp}
\DeclareMathOperator{\loc}{loc\ }
\DeclareMathOperator{\divg}{div}
\DeclareMathOperator{\dist}{dist}
\DeclareMathOperator{\osc}{osc\ }
\DeclareMathOperator{\capc}{cap}
\DeclareMathOperator{\curl}{curl}
\DeclareMathOperator{\diam}{diam}
\newtheorem{thm}{Theorem}[section]
\newtheorem{cor}{Corollary}[section]
\newtheorem{lem}[thm]{Lemma}
\newtheorem{prop}{Proposition}[section]
\newtheorem{defn}{Definition}[section]

\newtheorem{re}{Remark}[section]

\theoremstyle{remark}

\numberwithin{equation}{section}
\newcommand{\norm}[1]{\left\Vert#1\right\Vert}
\newcommand{\abs}[1]{\left\vert#1\right\vert}

\newcommand{\Real}{\mathbb R}
\newcommand{\Rn}{\mathbb R^n}

\newcommand{\vp}{\varphi}

\newcommand{\Di}{_{x_i}}
\newcommand{\Dj}{_{x_j}}
\newcommand{\bdy}{\partial}
\newcommand{\ep}{_{\epsilon}}

\begin{document}
\title[Boundary behavior of solutions]{Boundary behavior of solutions of elliptic operators in divergence form with a BMO anti-symmetric part}
\author{Linhan Li \and Jill Pipher}

\newcommand{\Addresses}{{
  \bigskip
  \footnotesize

   \textsc{Department of Mathematics, Brown University,
    Providence, RI 02906 USA}\par\nopagebreak
  \textit{E-mail address}, L.~Li: \texttt{linhan\_li@brown.edu}\quad J.~Pipher: \texttt{jpipher@math.brown.edu}

}}

\begin{abstract}In this paper, we investigate the boundary behavior of solutions of divergence-form operators with an elliptic symmetric part and a $BMO$ anti-symmetric part. Our results will hold in non-tangentially accessible (NTA) domains; these general domains were introduced by Jerison and Kenig and include the class of Lipschitz domains. We establish the H\"older continuity of the solutions at the boundary, existence of elliptic measures $\omega_L$ associated to such operators, and the well-posedness of the continuous Dirichlet problem as well as the $L^p(d\omega)$ Dirichlet problem in NTA domains. The equivalence in the $L^p$ norm of the square function and the non-tangential maximal function under certain conditions remains valid. When specialized to Lipschitz domains, it is then possible to extend, to these operators, various criteria for determining mutual absolute continuity of elliptic measure with surface measure.
\end{abstract}
\maketitle

\section{Introduction}
In this paper, we investigate the boundary behavior of solutions of divergence-form operators with an elliptic symmetric part and a $BMO$ anti-symmetric part. More precisely, we consider operators $L=-\divg A\nabla$, $A=a+b$, with $a$ symmetric, bounded and satisfying the usual ellipticity condition, and $b$ anti-symmetric and belonging to the space $BMO$ (bounded mean oscillation).These operators arise in the study of elliptic equations of the form
$$
-\Delta u+c\cdot\nabla u=f
$$
with a divergence-free drift $c$. Since $\divg c=0$, we can write $c=\divg b$ for an anti-symmetric tensor $b=(b_{ij})$, and the preceding equation becomes
$$
-\divg(I-b)\nabla u=f.
$$

When the matrix $A$ is bounded, measurable and elliptic, there are classical results about the boundary behavior of the solution to $Lu=0$ in a domain, $\Omega \subset \Rn$. These investigations were initiated in the papers of De Giorgi, Nash, and Moser, where sharp regularity of solutions to divergence form elliptic equations with bounded measurable coefficients were obtained.
In particular, (weak) solutions were shown be H\"older continuous, with a parameter depending only upon ellipticity, in domains satisfying certain exterior cone condition.
 The boundary regularity can be obtained essentially using Moser iteration, see \cite{gilbarg2001elliptic} for example. The study of regular points, as well as the solvability of the continuous Dirichlet problem for such operators was established in \cite{littman1963regular}. Then the fundamental work of \cite{hunt1968boundary} and \cite{caffarelli1981boundary} showed the existence of non-tangential limits of solutions, paving the way for the study of $L^p$ Dirichlet problem and other boundary value problems for divergence form elliptic operators with bounded and measurable coefficients. In  \cite{kenig2000new}, it was observed that the results of \cite{caffarelli1981boundary} are valid without the symmetry assumption on $A$. Readers interested in a more complete history of boundary value problems for elliptic operators in divergence form might consult the book \cite{kenig1994harmonic}, which provides a concise exposition of the developments in this subject. \par

A function in $BMO$ is not necessarily in $L^{\infty}$, and thus the operators under consideration do not belong to the well-studied class of elliptic operators. In an important development, Seregin, Silvestre, {\v{S}}ver{\'a}k, and Zlato{\v{s}} \cite{seregin2012divergence} discovered that a large portion of Moser's arguments works for this class of operators. In particular, they carried out the Moser iteration, and proved the Liouville theorem and Harnack inequality for solutions to divergence form parabolic operators, including the elliptic case. Thus, the interior regularity theory of De Giorgi, Nash, and Moser carries over to this setting. Moreover, in \cite{qian2016parabolic}, the authors showed the existence of the fundamental solution of the parabolic operator $L-\partial_t$, and derived Gaussian estimate for the fundamental solution. In \cite{escauriaza2017kato}, Escauriaza and Hofmann showed that the domain of the square root $\sqrt{L}$ contains $W^{1,2}(\Rn)$, and that $\norm{\sqrt{L}f}_{L^2(\Rn)}\lesssim\norm{\nabla f}_{L^2(\Rn)}$ holds over $\dot{W}^{1,2}(\Rn)$ without using the Gaussian estimate. Recently, H. Dong and S. Kim \cite{dong2017fundamental} have generalized the result for fundamental solutions to second-order parabolic systems, under the assumption that weak solutions of the system satisfy a certain local boundedness estimate. However, the study of the boundary behavior of solutions is not part of the literature.
It is thus natural to ask whether the classical results for elliptic operators with bounded and measurable coefficients in divergence form remain valid for these operators.

This paper is organized as the following. In Section \ref{preSec}, we set down some definitions and basic facts about the operator $L=-\divg A\nabla$ where the anti-symmetric part of $A$ belongs to the space $BMO$. In Section \ref{intSec}, we show the interior estimates and Harnack principle of solutions. Although most of the results in that section have already been obtained in \cite{seregin2012divergence} for parabolic equations, the proofs we give in the elliptic case are simpler than the parabolic case.  In fact, we prove these results directly for non-smooth coefficients, instead of passing through the smooth case  as in \cite{seregin2012divergence}. Section \ref{BdySection} is devoted to proving the boundary H\"older continuity of the solution and that the continuous Dirichlet problem is uniquely solvable on non-tangentially accessible (NTA) domains. This class of domains was first defined in \cite{jerison1982boundary} and includes the class of Lipschitz domains. In Section \ref{GreenSection}, we check that the classical results on the Green's function and on regular points in \cite{gruter1982green} hold in this setting. In Section \ref{ellipticSec}, we discuss the existence of elliptic measure and observe that the classical properties of these measures are preserved. In Section \ref{LpSec}, we show the existence and uniqueness of the solution of Dirichlet problem boundary data in $L^p$ with respect to the elliptic measure. Section \ref{squareSec} is mainly devoted to proving the identity \eqref{identity}, which is crucial to establishing the equivalence of the $L^p$ norm of the square function and the $L^p$ norm of the non-tangential maximal function under certain conditions. Then, we observe that criteria (\cite{kenig2000new},
\cite{kenig2016square}) for determining when elliptic measure and surface measure are mutually absolutely continuous on Lipschitz domains holds for this class of operators.\par

The first author would like to thank Seick Kim for pointing out the paper \cite{seregin2012divergence} and suggesting this investigation.

\section{Preliminaries}\label{preSec}

\begin{defn}\label{CorkscrewDefn}[Corkscrew condition] A domain $\Omega\subset\Rn$ is said to satisfy the interior corkscrew condition(resp. exterior Corkscrew condition) if there exists $M>1$ and $R>0$ such that for any $Q\in\bdy\Omega$ and any $0<r<R$, there exists a corkscrew point (or non-tangential point) $A=A_r(Q)\in\Omega$ (resp. $A\in\Omega^c$) such that
\begin{equation}
    \abs{A-Q}<r \text{ and } \delta(A)\doteq \dist(A,\bdy\Omega)>\frac{r}{M}.
\end{equation}
  
\end{defn}

\begin{defn}\label{HarnackChainDefn}[Harnack Chain] A domain $\Omega\subset\Rn$ is said to satisfy the Harnack chain condition if there are universal constants $m>1$ and $m'>0$ such that for every pair of points $A$ and $A'$ in $\Omega$ satisfying
\begin{equation}
    l\doteq\frac{\abs{A-A'}}{\min\{\delta(A),\delta(A')\}}>1,
\end{equation}
  there is a chain of open Harnack balls $B_1$, $B_2$,$\dots$, $B_N$ in $\Omega$ that connects $A$ to $A'$. Namely, $A\in B_1$, $A'\in B_N$, $B_i\cap B_{i+1}\neq\emptyset$ and 
  \begin{equation}
      m^{-1}\diam(B_i)\le\delta(B_i)\le m\diam(B_i) \qquad\forall\, i,
  \end{equation}
  and the number of balls $N\le m'\log_2 l$.
  
\end{defn}

\begin{defn}[NTA domain]
  We say a domain $\Omega\subset\Rn$ is an NTA domain if it satisfies the interior and exterior Corkscrew condition, and the Harnack chain condition.
\end{defn}

\begin{defn}[1-sided NTA domain] If $\Omega\subset\Rn$ satisfies only the interior Corkscrew condition, and the Harnack chain condition, then we say that it is a 1-sided NTA domain. 1-sided NTA domains are also called uniform domains.
\end{defn}

In the rest of the paper, by $\Omega$ we always mean an open, bounded NTA domain in $\Rn$, with $n\ge3$, unless otherwise stated. 

  Let $L=-\divg A\nabla$ be a divergence form operator. Write the $n\times n$ matrix $A$ as $A=a+b$, where $a$ is the symmetric part and $b$ is the anti-symmetric part. We assume that $a$ and $b$ satisfy following conditions:

  $a=(a_{ij}(x))$ is a matrix of real, bounded measurable functions on $\Omega$, with bound \begin{equation}\label{aBound}
  \norm{a}_{L^{\infty}(\Omega)}\le\Lambda,
  \end{equation}
  and there exists a $\lambda>0$ such that
  \begin{equation}\label{aEllip}
  (a\xi)\cdot\xi\ge\lambda\abs{\xi}^2 \qquad\forall\,\xi\in\Rn.
  \end{equation}
   $b=(b_{ij}(x))$ satisfies
   \begin{equation}\label{bBMO}
   b_{ij}(x)\in BMO(\Omega)\quad\text{with}\ \norm{b}_{BMO(\Omega)}\le\Gamma.
   \end{equation}

   Note that the anti-symmetry of $b$ implies that the ellipticity condition \eqref{aEllip} on $a$ is actually an ellipticity condition on $A$. That is,
   \begin{equation}\label{AEllip}
   (A\xi)\cdot\xi=(a\xi)\cdot\xi\ge\lambda\abs{\xi}^2 \qquad\forall\,\xi\in\Rn.
   \end{equation}

    Recall that we say $b\in BMO(\Omega)$ if
  \begin{equation}\label{BMOdefn}
  \norm{b}_{BMO(\Omega)}\doteq\sup_{Q\subset \Omega}\fint_{Q}\abs{b-(b)_{Q}}dx<\infty
  \end{equation}
  where the supremum is taken over all cubes $Q$ with sides parallel to the axes, and
  $(f)_Q=\frac{1}{\abs{Q}}\int_Qf(x)dx$.

  P. Jones showed in \cite{jones1980extension} that, every function $b\in BMO(\Omega)$ admits an extension to some $\tilde{b}\in BMO(\Rn)$ if and only if the domain is uniform. In particular, for an NTA domain $\Omega$, $b\in BMO(\Omega)$, there exists $\tilde{b}\in BMO(\Rn)$ such that
  \begin{equation}\label{BMOextension}
    \tilde{b}|_{\Omega}=b \quad\text{and } \norm{\tilde{b}}_{BMO(\Rn)}\le C\norm{b}_{BMO(\Omega)},
  \end{equation}
where the constant $C$ depends only on the domain and dimension.\par


  We will consider the following:
  \begin{enumerate}
  \item\label{Lu=0} the notion of weak solution of
  $$
      Lu=0 \quad\text{in } \Omega,
  $$
  \item\label{clssDirichlet} the solvability of the classical Dirichlet problem
  $$
  \begin{cases}
    Lu=0 \quad\text{in } \Omega,\\
    u-g\in W_0^{1,2}(\Omega) \quad \text{where }g\in W^{1,2}(\Omega),
  \end{cases}
  $$
  and
  \item\label{contDirichlet} the solvability of the continuous Dirichlet problem
  $$
  \begin{cases}
    Lu=0 \quad\text{in } \Omega,\\
    u=g  \quad\text{on }\bdy\Omega \quad\text{where } g\in C(\bdy\Omega).
  \end{cases}
  $$
  \end{enumerate}

  Let $B[\quad,\quad]$ be the bilinear form associated with $L$, i.e.
  $$
  B[u,v]=\int_{\Omega}(a\nabla u\cdot\nabla v +b\nabla u\cdot\nabla v)dx.
  $$
  Here and in the sequel we write
  \begin{equation}
   {\int_{\Omega}b\nabla u\cdot\nabla\vp}=\frac1{2}\int_{\Omega}b_{ij}(x)(u\Dj \vp\Di-u\Di \vp\Dj)dx
  \end{equation}

  An important observation in \cite{seregin2012divergence} is that the bilinear form $B[u,v]$ is bounded, namely,
  $$
  \abs{\int_{\Rn}b\nabla u\cdot\nabla vdx}\le C\norm{\nabla u}_{L^2(\Rn)}\norm{\nabla v}_{L^2(\Rn)}, \forall\, u,v\in \dot{W}^{1,2}(\Rn).
  $$
  We indeed have, 
      \begin{equation}\label{energyest}
      \abs{B[u,v]}\le C\norm{\nabla u}_{L^2(\Omega)}\norm{\nabla v}_{L^2(\Omega)}
      \end{equation}
      for all $u\in W_0^{1,2}(\Omega)$, $v\in W_0^{1,2}(\Omega)$, and 
      \begin{equation}\label{energyest1}
        \abs{B[g,v]}\le C\norm{g}_{W^{1,2}(\Omega)}\norm{\nabla v}_{L^2(\Omega)} 
      \end{equation}
       for all $g\in W^{1,2}(\Omega)$, $v\in W_0^{1,2}(\Omega)$
      
    To see \eqref{energyest}, it suffices to show
    \begin{equation}\label{formupperbound}
      \abs{\int_{\Omega}b\nabla u\cdot\nabla v}\le C\norm{\nabla u}_{L^2(\Omega)}\norm{\nabla v}_{L^2(\Omega)}.
    \end{equation}
    Let $\tilde{u}$ and $\tilde{v}$ be the zero extension to $\Rn$ of $u$ and $v$, respectively.  Set $B=\nabla \tilde{v}$ and $E=(0,\dots,\tilde{u}\Dj,0,\dots,-\tilde{u}\Di,0,\dots)$, whose $i$th component is $\tilde{u}\Dj$ and $j$th component is $-\tilde{u}\Di$. Observe that $E,B\in L^2(\Rn)^n$, $\divg E=0$ and $\curl B=0$ in $\mathscr{D'}(\Rn)$. So the div-curl lemma of \cite{coifman1993compensated} gives that $E\cdot B\in \mathscr{H}^1(\Rn)$, and
    $$
    \|E\cdot B\|_{\mathscr{H}^1(\Rn)}\le C \|E\|_{L^2(\mathbf{R}^n)}\|B\|_{L^{2}(\mathbf{R}^n)}.
    $$
 That is, $\tilde{u}\Dj\tilde{v}\Di-\tilde{u}\Di\tilde{v}\Dj\in \mathscr{H}^1(\Rn)$ and
    \begin{equation}\label{H1norm}
    \norm{\tilde{u}\Dj\tilde{v}\Di-\tilde{u}\Di\tilde{v}\Dj}_{\mathscr{H}^1(\Rn)}\le C\norm{\nabla\tilde{u}}_{L^2(\Rn)}\norm{\nabla\tilde{v}}_{L^2(\Rn)}.
    \end{equation}
    By letting $\tilde{b}\in BMO(\Rn)$ be the extension of $b$, we have
    \begin{align*}
      \abs{\int_{\Omega}b_{ij}(x)(u\Dj v\Di-u\Di v\Dj) dx}
      &=\abs{\int_{\Rn}\tilde{b}_{ij}(x)(\tilde{u}\Dj \tilde{v}\Di-\tilde{u}\Di \tilde{v}\Dj) dx}\\
      &\le C\norm{\tilde{b}}_{BMO(\Rn)}\norm{\nabla\tilde{u}}_{L^2(\Rn)}\norm{\nabla\tilde{v}}_{L^2(\Rn)}\\
      &\le C\norm{b}_{BMO(\Omega)}\norm{\nabla u}_{L^2(\Omega)}\norm{\nabla{v}}_{L^2(\Omega)},
    \end{align*}
    which proves \eqref{formupperbound}.
    
    For \eqref{energyest1}, the same argument works provided that there is a $\tilde{g}\in W^{1,2}(\Rn)$ such that $\tilde{g}|_{\Omega}=g$ and $\norm{\nabla\tilde{g}}_{L^2(\Rn)}\lesssim\norm{g}_{W^{1,2}(\Omega)}$, with the implicit constant only depends on the domain and dimension.
    In \cite{jones1981quasiconformal}, P. Jones showed that such extension is possible. In fact, he showed that if the domain $\Omega$ is locally uniform, or by the terminology of Jones, a $(\epsilon,\delta)$ domain, then there exists a bounded linear extension operator 
    $$
    E:\ W^{1,p}(\Omega)\rightarrow W^{1,p}(\Rn) 
    $$
    with $E|_{\Omega}g=g$ for all $g\in W^{1,p}(\Omega)$.

    We will use the following slight generalization of \eqref{H1norm}, also a consequence of the div-curl lemma.
    \begin{prop}\label{compensatedProp}
      If $u\in W^{1,p}(\mathbf{R}^n),\ v\in W^{1,p'}(\mathbf{R}^n)$, then for fixed $i $ and $j$,
      we have $u\Dj v\Di-u\Di v\Dj\in \mathscr{H}^1(\mathbf{R}^n)$, and
      $$
      \norm{u\Dj v\Di-u\Di v\Dj}_{\mathscr{H}^1(\mathbf{R}^n)}\le C\|\nabla u\|_{L^p(\mathbf{R}^n)}
      \|\nabla v\|_{L^{p'}(\mathbf{R}^n)}.
      $$
    \end{prop}




The upper bound \eqref{energyest} of the bilinear form enables us to define weak solutions of problem \eqref{Lu=0}--\eqref{contDirichlet} in the following sense.

  \begin{defn}
  We say that a function $u\in W_{loc}^{1,2}(\Omega)$ is a weak solution of \eqref{Lu=0} if
  \begin{equation}\label{s1}
  \int_{\Omega}(a\nabla u\cdot\nabla\vp +b\nabla u\cdot\nabla\vp) dx=0
  \end{equation}
  for all $\vp\in W^{1,2}(\Omega)$ with $\supp\vp\subset\subset\Omega$.\par
  We say that $u\in W_{loc}^{1,2}(\Omega)$ is a weak subsolution(supersolution) of \eqref{Lu=0} if
  \begin{equation}
  \int_{\Omega}(a\nabla u\cdot\nabla\vp +b\nabla u\cdot\nabla\vp) dx\le(\ge)0
  \end{equation}
  for all $\vp\in W^{1,2}(\Omega)$ with $\supp\vp\subset\subset\Omega$ and $\vp\ge0$ a.e. in $\Omega$.

  \end{defn}


  \begin{defn}
    We say a function $u\in W^{1,2}(\Omega)$ is a weak solution of \eqref{clssDirichlet} if
    $$
    \int_{\Omega}(a\nabla u\cdot\nabla\vp +b\nabla u\cdot\nabla\vp) dx=0 \quad \forall \vp\in W_0^{1,2}(\Omega),
    $$
    and
    \begin{equation}
      u-g\in W_0^{1,2}(\Omega).
    \end{equation}
  \end{defn}

 \begin{defn}
   We say a function $u\in W_{loc}^{1,2}(\Omega)\cap C(\overline{\Omega})$ is a weak solution of \eqref{contDirichlet} if \eqref{s1} holds for all $\vp\in W^{1,2}(\Omega)$ with $\supp\vp\subset\subset\Omega$, and
   $$
   u=g \quad \text{on }\bdy\Omega.
   $$
 \end{defn}

The ellipticity of $A$ \eqref{AEllip} immediately gives
 \begin{equation}\label{ellipticity}
   B[u,u]\ge\lambda\norm{\nabla u}^2_{L^2(\Omega)}\qquad\forall u\in W_0^{1,2}(\Omega).
 \end{equation}

 Therefore, together with \eqref{energyest}, we can apply the Lax-Milgram theorem to get
    \begin{thm}\label{LaxMilgramLem}
    Given $h\in W^{1,2}(\Omega)^*=W^{-1,2}(\Omega)$, there exists a unique $u\in W_0^{1,2}(\Omega)$ such that $Lu=h$, in the sense that
    $$
    \int_{\Omega}(a\nabla u\cdot\nabla\vp+ b\nabla u\cdot \nabla\vp)=\langle \vp,h\rangle, \quad\forall\, \vp\in W_0^{1,2}(\Omega).
    $$
    \end{thm}
    Writing $w=u-g\in W_0^{1,2}(\Omega)$ in the classical Dirichlet problem \eqref{clssDirichlet}, the estimate \eqref{energyest1} implies that $Lg\in W^{-1,2}(\Omega)$. So we can apply the theorem to conclude that the classical Dirichlet problem \eqref{clssDirichlet} is uniquely solvable.

\section{Interior estimates of the solution}\label{intSec}
Almost all the lemmas in this section appeared in \cite{seregin2012divergence} in the context of parabolic equations, while we specialize to the elliptic case. We include the proofs not merely for the sake of completeness, but because some of the arguments are completely different and lend themselves to the the development of the boundary regularity.

  \begin{lem}[Caccioppoli inequality]\label{LemIntCacci}
  Let $u\in W^{1,2}_{\loc}(\Omega)$ be a weak solution of \eqref{Lu=0}. 
  Let $B_R=B_R(x)$ be a ball centered at $x$ with radius $R$ such that $\overline{B_R}\subset\Omega$. Then for any $0<\sigma<1$ and $c$,
  \begin{equation}
    \int_{B_R}\abs{\nabla u}^2\vp^2\lesssim\frac1{(1-\sigma)^2R^2}\int_{\supp\nabla\vp}\abs{u-c}^2+
    \frac1{(1-\sigma)R}\int_{\supp\nabla\vp}\abs{b-(b)_R}\abs{\nabla u\vp}\abs{u-c},
  \end{equation}
  where $\vp\in C_0^{\infty}(B_R)$ is nonnegative, satisfying $\vp=1$ on $B_{\sigma R}$, $\supp\vp\subset B_R$, and $\abs{\nabla\vp}\le\frac{C}{(1-\sigma)R}$.
  \end{lem}

  \begin{proof}
    Take $\phi=(u-c)\vp^2$ as test function in \eqref{s1} to get
    \begin{equation}
      \int_{B_R}(a+b)\nabla u\cdot\nabla u\vp^2+\int_{B_R}(a+b)\nabla u\cdot\nabla\vp^2(u-c)=0.
    \end{equation}
    Using ellipticity of $a$ and anti-symmetry of $b$, we have
    $$
    \lambda\int_{B_R}\abs{\nabla{u}}^2\vp^2+\int_{B_R}a\nabla u\cdot\nabla\vp^2(u-c)+\int_{B_R}(b-(b)_R)\nabla u\cdot\nabla\vp^2(u-c)+\int_{B_R}(b)_R\nabla u\cdot\nabla\vp^2(u-c)\le0.
    $$
    Since $(b)_R$ is an anti-symmetric constant matrix and $\vp$ is smooth, the divergence theorem gives
    \begin{align}\label{zero}
    \int_{B_R}(b)_R\nabla u\cdot\nabla\vp^2(u-c)&=\frac1{2}\int_{B_R}(b)_R\nabla(u-c)^2\cdot\nabla\vp^2\nonumber\\
    &=-\frac1{2}\int_{B_R}\divg\Big((b)_R\nabla\vp^2\Big)(u-c)^2+\frac1{2}\int_{\bdy B_R}\nabla\vp^2\cdot \vec{N}(u-c)^2d\sigma\nonumber\\
    &=0.
    \end{align}

    So
    \begin{align*}
      &\lambda\int_{B_R}\abs{\nabla{u}}^2\vp^2\\
      &\le \Lambda(\int_{B_R}\abs{\nabla u\vp}^2)^{1/2}(\int_{B_R}\abs{u-c}^2\abs{\nabla\vp}^2)^{1/2}+2\int_{B_R}\abs{b-(b)_R}\abs{\nabla u\vp}\abs{u-c}\abs{\nabla\vp}.
    \end{align*}
    Then the desired result follows from Young's inequality.

  \end{proof}

  The following Caccioppoli type inequality will be used frequently.

  \begin{cor}\label{caccio'}
  Let $u$, $B_R$ be as in Lemma \ref{LemIntCacci}. Then for any $1<s\le\frac{n}{n-1}$,
    $$
    \int_{B_{R/2}}\abs{\nabla u}^2\le C(s,n,\Lambda,\lambda)(1+\Gamma^2)R^{2(\frac{n}{s'}-1)}(\int_{B_R}\abs{\hat u}^{\frac{2s}{2-s}})^{\frac{2-s}{s}},
    $$
    where $\hat{u}=u-(u)_R$ and $s'=\frac{s}{s-1}$.
    Equivalently,
    $$
    \fint_{B_{R/2}}\abs{\nabla u}^2\le C(s,n,\Lambda,\lambda)(1+\Gamma^2)R^{-2}(\fint_{B_R}\abs{\hat u}^{\frac{2s}{2-s}})^{\frac{2-s}{s}}
    $$
  \end{cor}

  \begin{proof}
    Taking $c=(u)_R$ and $\sigma=1/2$ in Lemma \ref{LemIntCacci}, we get
    \begin{align*}
    \int_{B_R}\abs{\nabla u}^2\vp^2&\le CR^{-2}\int_{B_R}\abs{\hat u}^2+
    CR^{-1}\int_{B_R}\abs{b-(b)_R}\abs{\nabla u\vp}\abs{\hat u}\\
    &\doteq I_1+I_2.
    \end{align*}
    By H{\"o}lder's inequality,
    $$
    I_1\le CR^{\frac{2n}{s'}-2}(\int_{B_R}\abs{\hat u}^{\frac{2s}{2-s}})^{\frac{2-s}{s}}.
    $$

    Using H{\"o}lder's inequality twice and then using Young's inequality,
    \begin{align*}
      I_2&\le CR^{\frac{n}{s'}-1}(\fint_{B_R}\abs{b-(b)_R}^{s'})^{\frac1{s'}}(\int_{B_R}\abs{\nabla u\vp}^s\abs{\hat u}^s)^{\frac1{s}}\\
      &\le C_sR^{\frac{n}{s'}-1}\Gamma(\int_{B_R}\abs{\nabla u\vp}^2)^{1/2}(\int_{B_R}\abs{\hat u}^{\frac{2s}{2-s}})^{\frac{2-s}{2s}}\\
      &\le\frac1{2}\int_{B_R}\abs{\nabla u\vp}^2+C_sR^{2(\frac{n}{s'}-1)}\Gamma^2(\int_{B_R}\abs{\hat u}^{\frac{2s}{2-s}})^{\frac{2-s}{s}}.
    \end{align*}
    Combining these two estimates we proved the corollary. Note that by Poincar\'e-Sobolev inequality,
    $\norm{\hat u}_{L^{\frac{2s}{2-s}}}\le\norm{u}_{W^{1,2}}$ for $1<s\le\frac{n}{n-1}$.
  \end{proof}

  \begin{re}
  The main difference from the usual Caccioppoli's inequality \eqref{Caccioineq} is that at this point, the power of $\tilde{u}$ on the right-hand side is greater than 2. The usual Caccioppli estimate will hold once we know that
  $$
  \Big(\fint_{B_R}\abs{u}^p\Big)^{1/p}\le \Big(\fint_{B_{2R}}\abs{u}^2\Big)^{1/2}  \quad\forall p>2.
  $$

   The recent paper \cite{dong2017fundamental} contains a somewhat simpler approach to \eqref{Caccioineq}, which is proven there for second-order parabolic systems. 
    
  \end{re}

  \begin{lem}[Reverse H\"older inequality of the gradiant]\label{RHgrad_int}
  Let $u$, $B_R$ be as in Lemma \ref{LemIntCacci}. Then there exist $p>2$ and $C$ depending only on $n$, $\lambda$, $\Lambda$ and $\Gamma$ such that
  \begin{equation}
    (\fint_{B_{R/2}}\abs{\nabla u}^p)^{\frac1{p}}\le C(\fint_{B_R}\abs{\nabla u}^2)^{\frac1{2}}.
  \end{equation}

  \end{lem}

  \begin{proof}
    Observe that the function $s\mapsto\frac{2s}{2-s}$ is strictly increasing on $(1,\frac{n}{n-1})$ with range $(2,\frac{2n}{n-2})$, and that the function $r\mapsto\frac{rn}{n-r}$ is strictly increasing on $(\frac{2n}{n+2},2)$ with range $(2,\frac{2n}{n-2})$. So for fixed $s\in(1,\frac{n}{n-1})$, there exists a unique $r\in(\frac{2n}{n+2},2)$ that solves $\frac{2s}{2-s}=\frac{rn}{n-r}$. \\
    Then by Corollary \ref{caccio'} and Poincar\'e-Sobolev inequality, we get
    $$
    \int_{B_{R/2}}\abs{\nabla u}^2\le C(s,n,\Lambda,\lambda)(1+\norm{b}^2_{BMO})R^n(\fint_{B_R}\abs{\nabla u}^r)^{\frac2{r}},
    $$
    that is
    $$
    (\fint_{B_{R/2}}\abs{\nabla u}^2)^{\frac1{2}}\le C(s,n,\lambda,
    \Lambda)(1+\Gamma^2)^{1/2}(\fint_{B_R}\abs{\nabla u}^r)^{\frac1{r}}.
    $$
    Then the result follows from Prop.1.1 of Chapter V of \cite{giaquinta1983multiple}.
  \end{proof}

  In \cite{seregin2012divergence}, interior regularity of the solution is first established for smooth coefficients case and in the general case via approximation arguments. The following approach gives estimates for solutions directly in the non-smooth case, which is advantageous for boundary estimates.\par

  Let us first point out a simple fact:

  \begin{lem}\label{pm_subLem}
    Let $u\in W_{loc}^{1,2}(\Omega)$ be a weak solution of \eqref{Lu=0}. Let $u^+=\max\{u,0\}$, $u^-=\max\{-u,0\}$. Then $u^+$ and $u^-$ are subsolutions of \eqref{Lu=0}, i.e.
    \begin{equation}\label{pm_subsolution}
  \int_{\Omega}(a\nabla u^{\pm}\cdot\nabla\vp +b\nabla u^{\pm}\cdot\nabla\vp) dx\le0
  \end{equation}
  for all $\vp\in W^{1,2}(\Omega)$ with $\supp\vp\subset\subset\Omega$ and $\vp\ge0$ a.e. in $\Omega$.
  \end{lem}
  \begin{proof}
  First consider $\vp\in C_0^{\infty}(\Omega)$. Consider $u^+,$ and for $k>1$, define
  $$
  \eta_k(t)=\left\{\begin{array}{rl}
0,&\quad t\le 0,\\
kt,&\quad 0<t\le \frac 1k,\\
1,&\quad t>\frac1 k.
\end{array}\right.
$$

Take $\eta_k(u)\vp$ as test function in \eqref{s1} and then let $k\rightarrow\infty$:
\begin{align*}
0&=\int_{\Omega}\Big(a\nabla u\cdot\nabla(\eta_k(u)\vp) +b\nabla u\cdot\nabla(\eta_k(u)\vp)\Big) dx\\
&\ge \int_{\Omega}a\nabla u\cdot\nabla\vp \eta_k(u) +\int_{\Omega}b\nabla u\cdot\nabla\vp \eta_k(u)dx
\rightarrow \int_{\Omega}\Big(a\nabla u^+\cdot\nabla\vp +b\nabla u^+\cdot\nabla\vp\Big) dx.
\end{align*}

For $\vp\in W^{1,2}(\Omega)$ with compact support in $\Omega$, the energy estimates \eqref{energyest} permits approximation by smooth functions.\par

 The argument for $u^-$ is similar.

  \end{proof}

 Let $q=\frac{2s}{2-s}$, $1<s<\frac{n}{n-1}$.

  \begin{lem}\label{Moser Lemma 1} Let $u\in W_{loc}^{1,2}(\Omega)$ be a weak solution of \eqref{Lu=0}. Assume $\overline{Q_{2R}}\subset\Omega$. Then for any $k\ge k_0>\frac1{2}$, we have
  \begin{equation}\label{moserest1}
    \sup_{Q_R}u^{\pm}\le C(n,\lambda,\Lambda,\Gamma,s,k_0)\Big(\fint_{Q_{2R}}(u^{\pm})^{kq}\Big)^{\frac{1}{kq}}.
  \end{equation}
  Therefore,
  \begin{equation}
   \sup_{Q_R}\abs{u}\le C(n,\lambda,\Lambda,\Gamma,s,k_0)(\fint_{Q_{2R}}\abs{u}^{kq})^{\frac{1}{kq}}.
  \end{equation}
  \end{lem}

  \begin{proof}
    We only show \eqref{moserest1} for $u^+$.\par
    $\forall\epsilon>0$, $N>>1$ and $\beta\ge k_0$, pick $\frac1{2}<k_1<\min\{1,k_0\}$. Define
    $$
    H_{\epsilon,N}(t)=
    \left\{\begin{array}{l l}
    t^\beta,&\quad t\in[\epsilon,N],\\
    N^\beta+\frac{\beta}{k_1} N^{\beta-k_1}(t^{k_1}-N^{k_1}), &\quad t>N.
    \end{array}\right.
   $$

   $$
   \Rightarrow  H'_{\epsilon,N}(t)=
    \left\{\begin{array}{rl}
   \beta t^{\beta-1},&\quad t\in (\epsilon,N),\\
   \beta N^{\beta-k_1}t^{k_1-1}, &\quad t>N.
   \end{array}\right.
   $$

   Set
   $$
   G_{\epsilon,N}(w)=\int_\epsilon^w|H'_{\epsilon,N}(t)|^2dt,\ w\ge \epsilon.
   $$

   For $w\ge \epsilon$, it is easy to check that
     \begin{equation}\label{H-G1}
       H(w)\le w^{\beta},
     \end{equation}
     \begin{equation} \label{H-G2}
       wH'(w)\le\beta w^{\beta},
     \end{equation}
      and
     \begin{equation}\label{H-G3}
      G(w)\le \frac1{2k_1-1}wG'(w).
     \end{equation}

    Here and in the sequel we omit the subscripts in $G_{\epsilon,N}$ and $H_{\epsilon,N}$.\par

    Set $u\ep(x)=u^+(x)+\epsilon$. Choose $\vp=G(u\ep)\eta^2$ as test function in \eqref{pm_subsolution}, where $\eta=1$ on $Q_{r'}$, $0\le \eta\le 1$, $\supp\eta\subset Q_r$ and $|\nabla \eta|\le \frac C{r-r'}$ for $R\le r'<r\le 2R$.\par
    Note that $\forall\beta\ge k_0$, $\vp\in W_0^{1,2}(\Omega)$, which is important for Moser iteration. In fact, for any $\beta\ge k_0$,
    $$
    \abs{H'(u\ep)}^2\in L^r(\Omega) \quad \forall r>1, \quad \Rightarrow
    $$

    $$
    \nabla G(u\ep)\in L^2(\Omega), \quad \abs{G(u\ep)}\le\frac1{2k_1-1}\abs{H'(u\ep)}^2u\ep\in L^2(\Omega).
    $$

    We compute (making use of \eqref{H-G2} and \eqref{H-G3})
    \begin{align}\label{H-Ga}
    \int_{Q_r}a\nabla u\ep\cdot\nabla\vp
    &=\int_{Q_r}a\nabla H(u\ep)\cdot\nabla H(u\ep)\eta^2+2\int_{Q_r}a\nabla u\ep\cdot \nabla\eta G(u\ep)\eta\nonumber\\
    &\ge \lambda\int_{Q_r}\abs{\nabla H(u\ep)}^2 \eta^2
    -\frac{2\Lambda}{2k_1-1}\int_{Q_r}\abs{\nabla u\ep}\abs{\nabla\eta}G'(u\ep)u\ep\abs{\eta}\nonumber\\
    &= \lambda\int_{Q_r}\abs{\nabla H(u\ep)}^2 \eta^2
    -\frac{2\Lambda}{2k_1-1}\int_{Q_r}\abs{\nabla H(u\ep)\eta}H'(u\ep)u\ep\abs{\nabla\eta}\nonumber\\
    &\ge \frac{\lambda}2 \int_{Q_r}|\nabla H(u\ep)|^2\eta^2-
     \frac C{((r-r')(2k_1-1))^2}\int_{Q_R}|u\ep H'(u\ep)|^2\nonumber\\
    &\ge \frac{\lambda}2\int_{Q_r}|\nabla H(u\ep)|^2\eta^2-\frac{C(\lambda,\Lambda,n)}{(2k_0-1)^2}\beta^2
     \frac{r^n}{(r-r')^2}(\fint_{Q_r}|u\ep|^{\beta q})^{\frac2{q}},
   \end{align}

   while
   \begin{align}\label{H-Gb}
     \abs{\int_{Q_r}(b\nabla u\ep)\cdot\nabla\vp}&
     =\abs{2\int_{Q_r}\Big(b-(b)_{Q_r}\Big)\nabla u\ep\cdot\nabla\eta G(u\ep)\eta}\nonumber\\
     &\le \frac C{(r-r')}\int_{Q_r}|b-(b)_{Q_r}|
\abs{\nabla u\ep}G(u^+_\epsilon)\eta\nonumber\\
 &\le \frac C{(r-r')(2k_1-1)}\int_{Q_r}\abs{b-(b)_{Q_r}}\abs{\nabla H(u\ep)}H'(u\ep)u\ep\eta\nonumber\\
 &\le \frac{\lambda}8\int_{Q_r}\abs{\nabla H(u\ep)}^2\eta^2+\frac {C(\lambda,n)}{\Big((2k_1-1)(r-r')\Big)^2}
\int_{Q_r}\abs{b-(b)_{Q_r}}^2\abs{H'(u\ep)u\ep}^2\nonumber\\
 &\le \frac{\lambda}8\int_{Q_r}\abs{\nabla H(u\ep)}^2\eta^2+\frac{C(n,s,\lambda)\Gamma^2r^n}{(2k_0-1)^2(r-r')^2}
\Big(\fint_{Q_r}|H'(u\ep)u\ep|^q\Big)^{\frac2{q}}\nonumber\\
 &\le \frac{\lambda}8\int_{Q_r}|\nabla H(u\ep)|^2\eta^2+\frac{C(n,s,\lambda)\Gamma^2\beta^2} {(2k_0-1)^2}\frac {r^n}{(r-r')^2}
\Big(\fint_{Q_r}|u\ep|^{\beta q}\Big)^{\frac2{q}}.
   \end{align}

   Then combining \eqref{pm_subsolution}, \eqref{H-Ga}, \eqref{H-Gb} and $r\sim r'$,
   $$
   \fint_{Q_{r'}}\abs{\nabla H(u\ep)}^2\le \frac{C(n,s,\lambda,\Lambda)(1+\Gamma^2)\beta^2}{(2k_0-1)^2}(r-r')^{-2}\Big(\fint_{Q_r}\abs{u\ep}^{\beta q}\Big)^{\frac2{q}}.
   $$

   By the Sobolev inequality,
   \begin{align*}
     (\fint_{Q_{r'}}H(u\ep)^{\frac{2n}{n-2}})^{\frac{n-2}n}
     &\le C\Big\{C(\lambda,\Lambda,\Gamma,n,s,k_0)\beta^2\frac{r'^2}{(r-r')^2}
\Big(\fint_{Q_r}|u\ep|^{\beta q}\Big)^{\frac2{q}}+\fint_{Q_{r'}}H^2(u\ep)\Big\}\\
&\le  C(\lambda,\Lambda,\Gamma,n,s,k_0)\beta^2\frac{r'^2}{(r-r')^2}(\fint_{Q_r}|u\ep|^{\beta q})^{\frac2{q}}.
   \end{align*}

   Letting $N\rightarrow \infty$,
$$
(\fint_{Q_{r'}}|u\ep|^{\beta \frac{2n}{n-2}})^{\frac{n-2}{2n}}\le C\beta\frac{r'}{r-r'}(\fint_{Q_r}
|u\ep|^{\beta q})^{\frac1{q}},
$$
that is, by setting $l=\frac{2n}{(n-2)q} (>1)$,
\begin{equation}\label{iteration0}
(\fint_{Q_{r'}}|u\ep|^{\beta lq})^{\frac1{lq}}\le C\beta\frac{r'}{r-r'}(\fint
_{Q_r}|u\ep|^{\beta q})^{\frac1{q}}.
\end{equation}

Now let $\beta=\beta_i=kl^i$, $r=r_i=R+\frac{R}{2^i}$ and $r'=r_{i+1}$ for $i=0,1,2,\dots$ in \eqref{iteration0}, one finds
  $$
  (\fint_{Q_{r_{i+1}}}u\ep^{kl^{i+1}q})^{\frac1{ql}}\le C(\lambda,\Lambda,\Gamma,s,n,k_0)kl^i2^i(\fint_{Q_{r_i}}u\ep^{kl^iq})^{\frac1{q}}
  $$

  $\Rightarrow$
  \begin{equation}
  (\fint_{Q_{r_{i+1}}}u\ep^{kl^{i+1}q})^{\frac1{ql^{i+1}}}\le (Ck(2l)^i)^{\sum_{j=0}^i\frac1{l^j}}(\fint_{Q_{2R}}u\ep^{kq})^{\frac1{q}}.
  \end{equation}

  Letting $i$ go to infinity, we get
  $$
  \sup_{Q_R}u\ep\le C(n,\lambda,\Lambda,\Gamma,s,k_0)(\fint_{Q_{2R}}u\ep^{kq})^{\frac{1}{kq}}.
  $$

  Finally, by letting $\epsilon$ go to 0 we get the desired estimate for $u^+$.

  \end{proof}

Observe that we now have the usual Caccioppli estimate:
  \begin{cor}[Caccioppoli]\label{caccio}
  Let $u\in W_{loc}^{1,2}(\Omega)$ be a non-negative weak subsolution of \eqref{Lu=0}.
  Let $B_R=B_R(X)$ be a ball centered at $X$ with radius $R$ such that $\overline{B_{2R}}\subset\Omega$. Then
  \begin{equation}\label{Caccioineq}
    \fint_{B_R}\abs{\nabla u}^2\le C(n,\lambda,\Lambda,\Gamma)R^{-2}\fint_{B_{2R}}u^2.
  \end{equation}

  \end{cor}

  \begin{lem}\label{Moser Lemma 2}
    Let $u\ge0$ be a weak supersolution of \eqref{Lu=0} and $\overline{B_{2R}}\subset\Omega$. Then for any $0<k<\frac1{2}$, we have
    \begin{equation}
    \sup_{B_R}u^k\le C(n,\lambda,\Lambda,\Gamma,s)(\fint_{B_{2R}}u^{kq})^{\frac{1}{q}}.
  \end{equation}
  \end{lem}

  \begin{proof}
  For any $\epsilon>0$, let $u\ep=u+\epsilon$, so that $u\ge\epsilon$. Set $v\ep=u\ep^k$, $\phi=ku\ep^{2k-1}\eta^2$, where $\eta\in C_0^{\infty}(B_{2R})$, $\supp\eta\subset B_r$, $\eta\equiv1$ on $B_{r'}$, $\abs{\nabla\eta}\le\frac{C}{r-r'}$ and $R\le r'<r\le 2R$. For convenience, we will omit the subscript $\epsilon$ in the following.\par
  Since $u$ is a supersolution, $\int A\nabla u\cdot\nabla\phi\ge0$, which gives
    \begin{equation}
    \frac{2k-1}{k}\int A\nabla v\cdot\nabla v\eta^2+2\int A\nabla v\cdot\nabla\eta\eta v\ge0.
    \end{equation}
     So
  \begin{align*}
    \frac{1-2k}{2k}\lambda\int\abs{\nabla v\eta}^2&\le
    \Lambda(\int\abs{\nabla v\eta}^2)^{1/2}(\int\abs{\nabla\eta v}^2)^{1/2}+\int\abs{b-(b)_r}\abs{\nabla v\eta}\abs{\nabla\eta v}\\
    &\le (\Lambda+C_s\Gamma) (\int\abs{\nabla v\eta}^2)^{1/2}(\int_{B_r}\abs{v}^q)^{1/q}(r-r')^{\frac{n}{s'}-1}
  \end{align*}
  $\Rightarrow$
  \begin{equation}
    \fint_{B_{r'}}\abs{\nabla v}^2\le C(\Lambda,\Gamma,\lambda,s,n)\theta^2(k)(r-r')^{-2}(\fint_{B_r}v^q)^{2/q},
  \end{equation}
  where $\theta(k)=\frac{2k}{2k-1}$.

  By Sobolev embedding, we have
  \begin{align}\label{iteration1}
    (\fint_{B_{r'}}v^{\frac{2n}{n-2}})^{\frac{n-2}{n}}&\le C(r'^2\fint_{B_{r'}}\abs{\nabla v}^2+\fint_{R_{r'}}v^2)\nonumber\\
    &\le C(\Lambda,\Gamma,\lambda,s,n)\Big(\theta^2(k)(\frac{r'}{r-r'})^2+1\Big)(\fint_{B_r}v^q)^{\frac2{q}}.
  \end{align}

    Recall that $l=\frac{2n}{(n-2)q}>1$, so there exists a $m\in\mathbb{Z}^+$ such that
    $\frac1{2}l^{-m+\frac1{2}}\le k\le \frac1{2}l^{-m+\frac5{2}}$. Denote $k'=\frac1{2}l^{-m+\frac1{2}}$. Then $l^{-2}k\le k'\le k$ and
    $$
    k'<k'l<k'l^2<\dots<k'l^{m-1}<\frac1{2}<k'l^m<\dots
    $$
    It is easy to check that $\theta^2(k'l^i)\le(l^{1/2}-1)^{-2}$ for $i=0,1,\dots, m-1$. \par
    Letting $v=u^{k'l^i}$ in \eqref{iteration1}, $i=0,1,\dots,m-1$, we get
    $$
    (\fint_{B_{r'}}u^{k'l^{i+1}q})^{\frac1{ql}}\le C(n,\lambda,\Lambda,\Gamma,s)\frac{r'}{r-r'}(\fint_{B_r}u^{k'l^iq})^{\frac1{q}}.
    $$
    Let $r=r_i=\frac{3R}{2}+\frac{R}{2^{i+1}}$ and $r'=r_{i+1}$.
    After $m$ iterations, we have
    $$
    (\fint_{B_{r_m}}u^{k'l^mq})^{\frac1{ql^m}}\le C(n,\lambda,\Lambda,\Gamma,s)(\fint_{B_{2R}}u^{k'q})^{\frac1{q}}.
    $$

    Since $k'l^m>\frac1{2}$, we can apply Lemma \ref{Moser Lemma 1} letting there $k=k_0=k'l^m$ and get
    \begin{align*}
      \sup_{B_R}u^{k'l^m}&\le C(n,\lambda,\Lambda,\Gamma,s)(\fint_{B_{3R/2}}u^{k'l^mq})^{\frac1{q}}\\
      &\le C(n,\lambda,\Lambda,\Gamma,s)^{1+l^m}(\fint_{B_{2R}}u^{k'q})^{\frac{l^m}{q}}.
    \end{align*}

    Raising to $\frac{k}{k'l^m}$ power on both sides and using H\"older inequality, we get
   $$
    \sup_{B_R}u^k\le C^{(l^{-m}+1)k/k'}(\fint_{B_{2R}}u^{k'q})^{\frac{k}{k'q}}
    \le C(n,\lambda,\Lambda,\Gamma,s)(\fint_{B_{2R}}u^{kq})^{\frac{1}{q}}.
   $$
   Finally, let $\epsilon\rightarrow 0$ we finish the proof.
  \end{proof}

  By Lemma \ref{Moser Lemma 1} and Lemma \ref{Moser Lemma 2}, we easily obtain the following
  \begin{cor}\label{MoserCor1}
    Let $u\ge0$ be a weak solution of \eqref{Lu=0} and $\overline{B_{2R}}\subset\Omega$. Then for any $p>0$, we have
    \begin{equation}\label{Harnack1}
    \sup_{B_R}u\le C(n,\lambda,\Lambda,\Gamma,p)(\fint_{B_{2R}}u^{p})^{\frac{1}{p}}.
    \end{equation}
  \end{cor}

  \begin{lem}\label{Moser Lemma 3}
    Let $u\ge0$ be a weak supersolution of \eqref{Lu=0}, $\overline{B_{2R}}\subset\Omega$. Then for any $k>0$,
    $$
    \sup_{B_R}u^{-k}\le C(n,\Lambda,\Gamma,\lambda,s)(\fint_{B_{2R}}u^{-kq})^{\frac1{q}}.
    $$
  \end{lem}

  \begin{proof}
    For $\epsilon>0$, define $u\ep=u+\epsilon$, $v\ep=u\ep^{-k}$ and $\phi\ep=ku\ep^{-2k-1}\eta^2$. Then $\phi\ep\in W_0^{1,2}(\Omega)$ for any $k>0$, and $u\ep$ is also a supersolution. So $\int A\nabla u\ep\cdot\nabla\phi\ep\ge0$, which implies
    $$
    \frac{2k+1}{k}\int A\nabla v\ep\eta\cdot\nabla v\ep\eta\le-2\int A\nabla u\ep\eta v\ep\nabla\eta.
    $$
    Arguing as in the proof of Lemma \ref{Moser Lemma 2}, we get
    \begin{align*}
      (\fint_{B_{r'}}v\ep^{\frac{2n}{n-2}})^{\frac{n-2}{n}}&\le C(\Lambda,\Gamma,\lambda,s,n)\Big((\frac{2k}{2k+1})(\frac{r'}{r-r'})^2
      +1\Big)(\fint_{B_r}v\ep^q)^{\frac2{q}}\\
    &\le C(\Lambda,\Gamma,\lambda,s,n)(\frac{r'}{r-r'})^2(\fint_{B_r}v\ep^q)^{\frac2{q}}.
    \end{align*}

    Applying iterations we obtain
    $$
    \sup_{B_R}u\ep^{-k}\le C(\Lambda,\Gamma,\lambda,s,n)(\fint_{B_{2R}}u\ep^{-kq})^{\frac1{q}}.
    $$
    Letting $\epsilon\rightarrow 0$, one finds the desired inequality for nonnegative solutions.
  \end{proof}

  \begin{cor}
    Let $u\ge0$ be a weak solution in $\Omega$. $\overline{B_{2R}}\subset\Omega$. Then for any $p<0$,
    \begin{equation}\label{Harnack2}
    \inf_{B_R}u\ge C(n,\lambda,\Lambda,\Gamma,p)(\fint_{B_{2R}}u^p)^{\frac{1}{p}}.
  \end{equation}
  \end{cor}

  \begin{proof}
    Since $\sup_{B_R}u^{-k}=(\inf_{B_R}u)^{-k}$, it follows immediately from the previous Lemma. Indeed,
    $$
    \inf_{B_R}u\ge C(n,\lambda,\Lambda,\Gamma,s)^{-\frac1{k}}(\fint_{B_{2R}}u^{-kq})^{\frac{1}{-kq}}
    $$
    for any $k>0$.
  \end{proof}

  Recall that we have the following result due to F.John and L.Nirenberg:
  \begin{lem}[\cite{john1961functions}]\label{John-Niren}
  Let $Q_1$ be the unit cube. Suppose $w\in BMO(Q_1)$, or equivalently, there exists some $0<C_1<\infty$ such that
  $$
  \fint_{Q}(w-(w)_Q)^2dx\le C_1^2 \quad \forall Q\subset Q_1,
  $$
  then there exist positive constants $\alpha,\beta$ depending on $n$ only such that
  $$
  \int_{Q_1}e^{\alpha\abs{w-(w)_{Q_1}}}dx\le\beta C_1.
  $$
  This implies that
  $$
  \int_{Q_1}e^{\alpha w}dx\int_{Q_1}e^{-\alpha w}dx\le\beta^2 C_1^2,
  $$
  and that we can assume $\alpha<\frac1{2}$.
  \end{lem}

  We are now ready to show the following.
  \begin{lem}[Harnack inequality]\label{Harnackineq}
  Let $u\ge0$ be a solution of \eqref{Lu=0}. $\overline{B_{2R}}\subset\Omega$. Then
  $$
  \sup_{B_{R}}u\le C(\Lambda,\lambda,\Gamma,n)\inf_{B_{R}}u.
  $$
  \end{lem}

  \begin{proof}
    As in Lemma \ref{Moser Lemma 3}, we first consider $u\ep=u+\epsilon$ then let $\epsilon$ go to 0. We omit these steps for simplicity. \par
    By scaling, we can assume $R=1$. Let $v=\log u$. We claim that $v\in BMO(B_{3/2})$. To see this, fix $B_r\subset B_{3/2}$. Let $\eta\in C_0^{\infty}(B_{3/2})$, $\supp\eta\subset B_{4r/3}$, $\eta\equiv1$ in $B_r$ and $\abs{\nabla\eta}\le\frac{C}{r}$. Define $\phi=\frac1{u}\eta^2$. Then
    $$
    0=\int A\nabla u\cdot\nabla\phi=-\int A\nabla v\cdot\nabla v\eta^2+2\int A\nabla\eta\cdot\nabla v\eta.
    $$
    $\Rightarrow$
    \begin{align*}
      \lambda\int\abs{\nabla v}^2\eta^2&\le 2\Lambda\int\abs{\nabla\eta}\abs{\nabla v\eta}
      +2\int\abs{b-(b)_{B_{4r/3}}}\abs{\nabla\eta}\abs{\nabla v\eta}\\
      &\le2\Lambda r^{\frac{n}{2}-1}(\int\abs{\nabla v\eta}^2)^{1/2}+Cr^{\frac{n}{2}-1}(\fint_{B_{4r/3}}\abs{b-(b)_{B_{4r/3}}}^2)^{1/2}
      (\int\abs{\nabla v\eta}^2)^{1/2},
    \end{align*}
    $\Rightarrow$
    $$
    \int_{B_r}\abs{\nabla v}^2\le C(\Lambda,\lambda,\Gamma,n)r^{n-2}.
    $$
    By the Poincar{\'e} inequality,
    $$
    \fint_{B_r}\abs{v-(v)_{B_r}}^2\le C_1^2(\Lambda,\lambda,\Gamma,n).
    $$

    Since $B_r\subset B_{3/2}$ is arbitrary, $v\in BMO(B_{3/2})$ with $\norm{v}_{BMO}\le C_1$. So we can apply Lemma \ref{John-Niren} to get
    \begin{equation}\label{John-Niren1}
      \int_{B_{3/2}}e^{\alpha v}\int_{B_{3/2}}e^{-\alpha v}\le(\beta C_1)^2,
    \end{equation}
    for some positive $\alpha$ and $\beta$ which only depend on $n$.\par
    Then from \eqref{John-Niren1}, \eqref{Harnack1} and \eqref{Harnack2}, Harnack's inequality follows.
  \end{proof}

  \begin{lem}[interior H{\"o}lder continuity]\label{intHolderLem}
   Let $u$ be a weak solution of \eqref{Lu=0}, and assume that $\overline{B_R(x)}\subset\Omega$. Then
  \begin{enumerate}
    \item For $0<\rho<r<R$, there exists $\frac1{2}<\theta=\theta (\Lambda,\lambda,\Gamma,n)<1$ and $\alpha=-\log_2\theta$ such that
    \begin{equation}
      \omega(\rho)\le\theta^{-1}(\frac{\rho}{r})^{\alpha}\omega(r),
    \end{equation}
    where $\omega(r)=\sup_{B_r(x)}u-\inf_{B_r(x)}u$ is the oscillation of $u$ in the ball with radius $r$.
    \item For $0<\rho<r<R/2$,
    \begin{equation}
      \omega(\rho)\le C(\Lambda,\lambda,\Gamma,n)(\frac{\rho}{r})^{\alpha}(\fint_{B_{2r}}u^2)^{1/2}.
    \end{equation}
  \end{enumerate}

  \end{lem}

 \begin{proof}
   By Harnack's inequality and a standard argument, we obtain
   $$
   \omega(r/2)\le \frac{C-1}{C+1}\omega(r)
   $$
   for $0<r<R$, where $C=C(\Lambda,\lambda,\Gamma,n)$ is the same constant as in Lemma \ref{Harnackineq}.
   Let $\theta=\frac{C-1}{C+1}<1$. Iteration gives
   $$
   \omega(r2^{-k})\le \theta^k\omega(r)  \qquad \text{for }k=1,2,\dots,
   $$
   which leads to (1).\par
   For (2), observe that
   $$
   \omega(r)=\sup_{B_r}(u^+-u^-)-\inf_{B_r}(u^+-u^-)
   \le 2\sup_{B_r}\abs{u}\le C(\fint_{B_{2r}}\abs{u}^2)^{1/2}.
   $$

 \end{proof}



\section{Estimates of the solution on the boundary}\label{BdySection}
\begin{lem}[boundary Caccioppoli]\label{LemBdyCaccio}
Let $P\in\bdy\Omega$, and let $T_r(P)=B_r(P)\cap\Omega$, $\Delta_r(P)=B_r(P)\cap\bdy\Omega$.
Let $u\in W^{1,2}(T_{R}(P))$ be a solution of \eqref{Lu=0} in $T_{R}(P)$, $u\equiv0$ on $\Delta_{R}(P)$.

  Then for any $0<\sigma<1$,
  \begin{align}
    \int_{T_R(P)}\abs{\nabla u}^2\vp^2
\le &C(n,\lambda,\Lambda)\frac1{(1-\sigma)^2R^2}\int_{\supp\nabla\vp\cap T_R(P)}\abs{u}^2\nonumber\\
    &+C(n,\lambda)\frac1{(1-\sigma)R}\int_{\supp\nabla\vp}\abs{\tilde{b}-(\tilde{b})_{B_R}}\abs{\nabla \tilde{u}\vp}\abs{\tilde{u}},
  \end{align}
  where $\vp\in C_0^{\infty}(B_R)$ is nonnegative, satisfying $\vp=1$ on $B_{\sigma R}$, $\supp\vp\subset B_R$, and $\abs{\nabla\vp}\le\frac{C}{(1-\sigma)R}$, and $\tilde{b}$ is the extension of $b$ as in \eqref{BMOextension}, $\tilde{u}$ is the zero extension of $u$ to $B_R(P)$.
  \end{lem}
  \begin{proof}
    Since $u=0$ on $\Delta_{R}(P)$, $u\vp^2\in W_0^{1,2}(T_R(P))$. Then one can proceed as in the proof of Lemma \ref{LemIntCacci} and get (denote $T_R(P)$ by $T_R$ and $B_R(P)$ by $B_R$)
    \begin{align*}
      &\lambda\int_{T_R}\abs{\nabla u}^2\vp^2\\
      &\le 2\Lambda(\int_{T_R}\abs{\nabla u}^2\vp^2)^{1/2}(\int_{T_R}\abs{u}^2\abs{\eta}^2)^{1/2}-2\int_{T_R}b\nabla u\cdot\nabla\vp\vp u\\
      &=2\Lambda(\int_{T_R}\abs{\nabla u}^2\vp^2)^{1/2}(\int_{T_R}\abs{u}^2\abs{\nabla\vp}^2)^{1/2}
      -2\int_{B_R}(\tilde{b}-(\tilde{b})_{B_R})\nabla\tilde{u}\cdot\nabla\vp\vp\tilde{u}\\
      &\le \frac{\lambda}{2}\int_{T_R}\abs{\nabla u}^2\vp^2+\frac{C(n,\Lambda,\lambda)}{(1-\sigma)^2R^2}\int_{\supp\nabla\vp\cap T_R}\abs{u}^2+
      \frac{C(n)}{(1-\sigma)R}
      \int_{\supp\nabla\vp}\abs{\tilde{b}-(\tilde{b})_{B_R}}\abs{\nabla\tilde{u}\vp}\abs{\tilde{u}},
    \end{align*}

    where we used
    $$
    \int_{B_R}(\tilde{b})_{B_R}\nabla\tilde{u}\cdot\nabla\vp\vp\tilde{u}=0
    $$
    as in \eqref{zero}.
  \end{proof}

   The boundary counterpart of Corollary \ref{caccio'} is the following.

   \begin{cor}\label{corbdyCaccio}
   Let $u$, $B_R(P)$, $T_R(P)$ be as in Lemma \ref{LemBdyCaccio}. Then for any $1<s\le\frac{n}{n-1}$,
    $$
    \int_{T_{R/2}(P)}\abs{\nabla u}^2\le C(s,n,\Lambda,\lambda)(1+\Gamma^2)R^{2(\frac{n}{s'}-1)}(\int_{T_R(P)}\abs{ u}^{\frac{2s}{2-s}})^{\frac{2-s}{s}},
    $$
    where $s'=\frac{s}{s-1}$.
    Equivalently,
    $$
    \fint_{T_{R/2(P)}}\abs{\nabla u}^2\le C(s,n,\Lambda,\lambda)(1+\Gamma^2)R^{-2}(\fint_{T_R(P)}\abs{u}^{\frac{2s}{2-s}})^{\frac{2-s}{s}}
    $$
  \end{cor}

  \begin{proof}
    Let $\tilde{b}$, $\tilde{u}$ be as in Lemma \ref{LemBdyCaccio}, and let $\sigma=1/2$. Then we have
    \begin{align*}
    \int_{T_R}\abs{\nabla u}^2\vp^2&\le CR^{-2}\int_{T_R}\abs{u}^2+
    CR^{-1}\int_{B_R}\abs{\tilde{b}-(\tilde{b})_R}\abs{\nabla \tilde u\vp}\abs{\tilde u}\\
    &\doteq I_1+I_2.
    \end{align*}

    \begin{align*}
    I_1&=CR^{-2}\int_{B_R}\abs{\tilde{u}}^2
    \le
    CR^{\frac{2n}{s'}-2}(\int_{B_R}\abs{\tilde{u}}^{\frac{2s}{2-s}})^{\frac{2-s}{s}}\\
    &=CR^{\frac{2n}{s'}-2}(\int_{T_R}\abs{u}^{\frac{2s}{2-s}})^{\frac{2-s}{s}}
    \end{align*}

    And
    \begin{align*}
      I_2&\le CR^{\frac{n}{s'}-1}(\fint_{B_R}\abs{\tilde{b}-(\tilde{b})_R}^{s'})^{\frac1{s'}}(\int_{B_R}\abs{\nabla \tilde u\vp}^s\abs{\tilde u}^s)^{\frac1{s}}\\
      &\le C_sR^{\frac{n}{s'}-1}\Gamma(\int_{B_R}\abs{\nabla\tilde u\vp}^2)^{1/2}(\int_{B_R}\abs{\tilde u}^{\frac{2s}{2-s}})^{\frac{2-s}{2s}}\\
      &=C_sR^{\frac{n}{s'}-1}\Gamma(\int_{T_R}\abs{\nabla u\vp}^2)^{1/2}(\int_{T_R}\abs{u}^{\frac{2s}{2-s}})^{\frac{2-s}{2s}}\\
      &\le\frac1{2}\int_{T_R}\abs{\nabla u\vp}^2+C_sR^{2(\frac{n}{s'}-1)}\Gamma^2(\int_{T_R}\abs{u}^{\frac{2s}{2-s}})^{\frac{2-s}{s}}.
    \end{align*}
    Combining these two estimates proves the corollary. Note that since $u$ vanishes on $\Delta_R(P)$,
    $\norm{u}_{L^{\frac{2s}{2-s}}(T_R)}\le\norm{u}_{W^{1,2}(T_R)}$ for $1<s\le\frac{n}{n-1}$.
  \end{proof}

  We will also need the analog of Lemma \ref{RHgrad_int} near the boundary:
  \begin{lem}
    Let $f\in Lip(\overline{\Omega})$, and $u$ be the weak solution of
    $$
    \begin{cases}
      Lu=-\divg A\nabla f \qquad\text{in }\Omega, \\
      u=0 \qquad\text{on }\bdy\Omega.
    \end{cases}
    $$
    Then there exists a $p>2$ and a constant $C$ depending on $n,\lambda,\Lambda,\Gamma$ and the domain, such that
    \begin{equation}
      \Big(\fint_{T_{\frac{R}{2}}(P)}\abs{\nabla u}^p\Big)^{\frac1{p}}
      \le C\Big\{\Big(\fint_{T_{R}(P)}\abs{\nabla u}^2\Big)^{\frac1{2}}+\norm{\nabla f}_{L^{\infty}(T_R(P))}\Big\},
    \end{equation}
    where $T_R(P)$ is defined as in Lemma \ref{LemBdyCaccio}.
  \end{lem}

  \begin{proof}
     For fixed $P\in\bdy\Omega$ and $0<R<\diam{\Omega}$, let $F\in Lip(\Rn)$ be the extension of $f$ such that 
\begin{equation}\label{lipext}
    \norm{\nabla F}_{L^{\infty}(\Rn)}\le C_n\norm{\nabla f}_{L^{\infty}(T_R(P))}
\end{equation}
where the constant $C_n$ is independent of $T_R(P)$ (see e.g. \cite{stein1970singular} p.174 Theorem 3).

    Let $\vp\in C^{\infty}_0(B_{2R/3}(P))$ be nonnegative, with $\vp\equiv1$ on $B_{\frac{R}{2}}(P)$ and $\abs{\nabla\vp}\lesssim R^{-1}$. 
    
    Let $\eta\in C^{\infty}_0(B_{R}(P))$ be nonnegative, with $\eta\equiv1$ on $B_{\frac{2R}{3}}(P)$ and $\abs{\nabla\eta}\lesssim R^{-1}$.
    
    Choose $u\vp^2\in W_0^{1,2}(T_{2R/3}(P))$ as test function to get
    $$
    \int_{T_R}A\nabla u\cdot\nabla(u\vp^2)=\int_{T_R}A\nabla f\cdot\nabla(u\vp^2).
    $$
    $\Rightarrow$
    \begin{align}\label{lem4.2ineq}
    \int_{T_R}a\nabla u\cdot\nabla u\vp^2
    &\le
    -2\int_{T_R}A\nabla u\cdot\nabla\vp u\vp+\int_{T_R}a\nabla f\cdot\nabla u\vp^2+\int_{T_R}a\nabla f\cdot\nabla\vp u\vp+\int_{T_R}b\nabla f\cdot\nabla(u\vp^2)\nonumber\\
    &\doteq I_1+I_2+I_3+I_4.
    \end{align}
    $I_1$ can be estimated as in the proof of Corollary \ref{corbdyCaccio}.
    For $I_2$ and $I_3$, we have
    $$
    \abs{I_2}\le\frac{\lambda}{8}\int_{T_R}\abs{\nabla u}^2\vp^2+C(n,\lambda,\Lambda)\int_{T_R}\abs{\nabla f}^2,
    $$
    and
    $$
    \abs{I_3}\le\Lambda\int_{T_R}\abs{\nabla f}^2+C(n,\Lambda)\int_{T_R}u^2\abs{\nabla\vp}^2.
    $$

    For $I_4$, 
    let $\tilde{F}=F-F(P)$. Extending $u\vp^2$ to be zero outside of $T_R$, $b$ to $\tilde{b}\in BMO(\Rn)$, we have
    $$
    I_4=\int_{\Rn}\tilde{b}\nabla(\tilde{F}\eta)\cdot\nabla(u\vp^2).
    $$
    $\Rightarrow$
    \begin{align*}
    \abs{I_4}&\le C(n,\Omega)\Gamma\Big(\int_{\Rn}\abs{\nabla(\tilde{F}\eta)}^2\Big)^{1/2}\Big(\int_{\Rn}\abs{\nabla(u\vp^2)}^2\Big)^{1/2}\\
    &\le C(n,\lambda,\Omega)\Gamma^2\Big(\int_{\Rn}\abs{\nabla \tilde{F}}^2\eta^2+\int_{\Rn}\tilde{F}^2\abs{\nabla{\eta}}^2\Big)
    +\frac{\lambda}{8}\Big(\int_{T_R}\abs{\nabla u}^2\vp^2+\int_{T_R}u^2\abs{\nabla\vp}^2\Big).
    \end{align*}
    
    Note that by \eqref{lipext},
    \begin{equation*}
       \int_{\Rn}\abs{\nabla \tilde{F}}^2\eta^2\le C_n\int_{B_R}\norm{\nabla F}^2_{L^{\infty}(B_R)}
       \le C_n\frac{\abs{B_R}}{\abs{T_R}}\int_{T_R}\norm{\nabla f}^2_{L^{\infty}(T_R)},
    \end{equation*}
    and
    \begin{align*}
        \int_{\Rn}\tilde{F}^2\abs{\nabla{\eta}}^2&\le \frac{C_n}{R^2}\int_{B_R}\abs{F-F(P)}^2\\
        &\le C_n\int_{B_R}\norm{\nabla F}^2_{L^{\infty}(B_R)}\le C_n\frac{\abs{B_R}}{\abs{T_R}}\int_{T_R}\norm{\nabla f}^2_{L^{\infty}(T_R)}.
    \end{align*}
    
    By the interior Corkscrew condition, $\frac{\abs{B_R}}{\abs{T_R}}\le M^n$, where $M$ is the constant in Definition \ref{CorkscrewDefn}. 
     Therefore,
    $$
    \abs{I_4}\le C(n,\lambda,M,\Omega)\Gamma^2\int_{T_R}\norm{\nabla f}_{L^{\infty}(T_R)}^2+
    \frac{\lambda}{8}\Big(\int_{T_R}\abs{\nabla u}^2\vp^2+\int_{T_R}u^2\abs{\nabla\vp}^2\Big).
    $$
    Combining estimates for $I_1$, $I_2$, $I_3$,$I_4$, and the inequality \eqref{lem4.2ineq}, we have
    $$
    \frac{\lambda}{2}\int_{T_R}\abs{\nabla u}^2\vp^2\le
    C(n,\lambda,\Lambda)(1+\Gamma^2)
    R^{2(\frac{n}{s'}-1)}\Big(\int_{T_R}\abs{u}^{\frac{2s}{2-s}}\Big)^{\frac{2-s}{s}}
    +C(n,\lambda,\Lambda,\Omega)(1+\Gamma^2)\int_{T_R}\norm{\nabla f}_{L^{\infty}(T_R)}^2,
    $$
    $\Rightarrow$
    $$
    \fint_{T_{\frac{R}{2}}}\abs{\nabla u}^2\le C(1+\Gamma^2)R^{-2}\Big(\fint_{T_R}\abs{u}^{\frac{2s}{2-s}}\Big)^{\frac{2-s}{s}}
    +C(1+\Gamma^2)\fint_{T_R}\norm{\nabla f}_{L^{\infty}(T_R)}^2.
    $$

    By the exterior Corkscrew condition, $\abs{B_R(P)\setminus\Omega}\ge (\frac1{2M})^n\abs{B_R(P)}$. So the assumption that $u=0$ on $\bdy\Omega$ implies
    $$
    R^{-2}\Big(\fint_{T_R}\abs{u}^{\frac{2s}{2-s}}\Big)^{\frac{2-s}{s}}\le C\Big(\fint_{T_R}\abs{\nabla u}^{r}\Big)^{\frac{2}{r}}
    $$
    for some $r<2$ (note that $1<s<\frac{n}{n-1}$). That is,
    $$
    \fint_{T_{\frac{R}{2}}}\abs{\nabla u}^2\le C(1+\Gamma^2)\Big(\fint_{T_R}\abs{\nabla u}^{r}\Big)^{\frac{2}{r}}+C(1+\Gamma^2)\fint_{T_R}\norm{\nabla f}_{L^{\infty}(T_R)}^2.
    $$
    Then the lemma follows from Proposition 1.1, Chap. V of \cite{giaquinta1983multiple}.
  \end{proof}

\begin{re}\label{RHgrad_global}
  Extending $u$ to be zero outside $\Omega$, one can see
  $$
  \Big(\int_{\Omega}\abs{\nabla u}^p\Big)^{\frac1{p}}
      \le C\Big\{\Big(\int_{\Omega}\abs{\nabla u}^2\Big)^{\frac1{2}}+\norm{\nabla f}_{L^{\infty}(\Omega)}\Big\}
  $$
  for some $p>2$, with $p$ and $C$ depending on $n,\lambda,\Lambda,\Gamma$ and the domain. On the other hand, we have
  $$
  \Big(\int_{\Omega}\abs{\nabla u}^2\Big)^{\frac1{2}}\le C(n,\Lambda,\Gamma)\norm{\nabla f}_{L^2(\Omega)}\le C(n,\Lambda,\Gamma, diam\,\Omega)\norm{\nabla f}_{L^{\infty}(\Omega)}.
  $$
  Therefore,
  \begin{equation}
    \Big(\int_{\Omega}\abs{\nabla u}^p\Big)^{\frac1{p}}\le C\norm{\nabla f}_{L^{\infty}(\Omega)}
  \end{equation}
  for some $p>2$.
\end{re}

  \begin{lem}\label{BdyHolderLem1}
    Let $u$ be a $W^{1,2}(\Omega)$ subsolution in $\Omega$. Then for any $P\in\Rn$, $R>0$ and $p>1$, we have
    \begin{equation}\label{BdyHolderSup}
      \sup_{B_R(P)}u_M^+\le C(n,\Lambda,\lambda,\Gamma,p)\Big(\fint_{B_{2R}(P)}{u_M^+}^p\Big)^{1/p},
    \end{equation}

    where
    $$
    M=\sup_{\bdy\Omega\cap B_{2R}(P)}u,
    $$

    $$
    u_M^+(x)=
    \begin{cases}
      \max\{u(x),M\}, \quad x\in\Omega,\\
      M, \qquad x\notin\Omega.
    \end{cases}
    $$
  \end{lem}

  \begin{proof}
    Note that the truncation $u_M$ is also a subsolution in $\Omega$. So if $B_R(P)$ is contained in $\Omega$, then \eqref{BdyHolderSup} is obtained by Lemma \ref{Moser Lemma 1}.
    Let $T_R(P)=B_R(P)\cap\Omega$. In the sequel we omit the point $P$ when no confusion is caused.\par
    Let
    $$
    U_M(x)=
    \begin{cases}
      \Big(u(x)-M\Big)^+ \quad x\in \Omega\\
      0  \qquad x\notin\Omega.
    \end{cases}
    $$

    Note that $U_M=0$ on $\Delta_{2R}$ and that
    \begin{equation}\label{UMu}
    u_M^+=U_M+M.
    \end{equation}
    Moreover, for any $\psi\in W_0^{1,2}(\Omega)$ with $\psi\ge0$ a.e.,
    $$
    \int_{\Omega}(a+b)\nabla U_M\cdot\nabla\psi\le0,
    $$
    which can be verified as in Lemma \ref{pm_subLem}.\par
    \textbf{Claim.} $\forall\quad k\ge k_0>\frac1{2}$,
    $$
    \sup_{T_R}U_M\le C(n,\lambda,\Lambda,\Gamma,s,k_0)\Big(\fint_{B_{2R}}U_M^{kq}\Big)^{\frac1{q}}.
    $$
    The same argument in the proof of Lemma \ref{Moser Lemma 1} gives this claim. In fact, for any $\epsilon>0$, define $U\ep=U_M+\epsilon$. Define $H_{\epsilon,N}$, $G_{\epsilon,N}$ and $\vp$ as in the proof of Lemma \ref{Moser Lemma 1}. Then one only need to note that
    $$
    \int_{T_{2R}}a\nabla U\ep\cdot\nabla\vp=\int_{B_r}\tilde{a}\nabla U\ep\cdot\nabla\vp,
    $$
    and
    $$
    \int_{T_{2R}}b\nabla U\ep\cdot\nabla\vp=\int_{B_r}\tilde{b}\nabla U\ep\cdot\nabla\vp,
    $$
    where
    $$
    \tilde{a}_{ij}(x)=
    \begin{cases}
      a_{ij}(x) \quad  x\in\Omega\\
      \lambda\delta_{ij}\quad x\notin\Omega,
    \end{cases}
    $$
    and $\tilde{b}\in BMO(\Rn)$ is the extension of $b$.\par

    By \eqref{UMu}, $\Big(\fint_{B_{2R}}U_M^{kq}\Big)^{\frac1{kq}}\le\Big(\fint_{B_{2R}}(u_M^+)^{kq}\Big)^{\frac1{kq}}$.
    Also, $M\le \Big(\fint_{B_{2R}}(u_M^+)^{kq}\Big)^{\frac1{kq}}$.
    Thus
    $$
    \sup_{B_R}u_M^+\le\sup_{B_R}U_M+M\le C\Big(\fint_{B_{2R}}(u_M^+)^{kq}\Big)^{\frac1{kq}}.
    $$

    Then \eqref{BdyHolderSup} follows.

  \end{proof}

 Unlike the interior case, we do not have the result for all $p>0$. This is mainly because when $0<k<\frac1{2}$, we have to modify our test function. However, we have the following result, which is the key to prove boundary H{\"o}lder continuity of the solution.\par
  Recall that $q=\frac{2s}{2-s}$, $1<s<\frac{n}{n-1}$, $\theta(k)=\abs{\frac{2k}{2k-1}}$.

  \begin{lem}\label{BdyHOlderLem2}
    Let $u$ be a $W^{1,2}(\Omega)$ supersolution in $\Omega$ which is non-negative in $\Omega\cap B_{4R}(P)$ for some ball $B_{4R}(P)\subset\Rn$. Then for any $p$ such that $0<p<\frac{n}{n-2}$,
    \begin{equation}
      \inf_{B_R}u_m^-\ge C(n,\Lambda,\lambda,\Gamma,p)\Big(\fint_{B_{2R}}{u_m^-}^p\Big)^{1/p},
    \end{equation}
    where
    $$
    m=\inf_{\bdy\Omega\cap B_{4R}}u,
    $$

    $$
    u_m^-(x)=
    \begin{cases}
      \min\{u(x),m\}, \quad x\in\Omega,\\
      m, \qquad x\notin\Omega.
    \end{cases}
    $$

  \end{lem}

  \begin{proof}
    For any $\epsilon>0$, let $u\ep=u_m^-+\epsilon$ and $v=u\ep^k$. As test function we choose
    $$
    \vp=k(u\ep^{2k-1}-(m+\epsilon)^{2k-1})\eta^2,
    $$
    where $\eta\in C_0^{\infty}(B_{2R})$, $\supp\eta\subset B_r$, $\eta\equiv1$ in $B_{r'}$, $\abs{\nabla\eta}\le\frac C{r-r'}$ and $R\le r'<r\le 2R$.\par
    Observe that for $k<\frac{1}{2}$, $\vp\in W_0^{1,2}(T_{4R})$, and that
    \begin{equation}\label{u-m}
      0\le u\ep^{2k-1}-(m+\epsilon)^{2k-1}\le u\ep^{2k-1}.
    \end{equation}

    Claim 1. For any $0<p<p_1<\frac{n}{n-2}$,
    \begin{equation}
      (\fint_{B_R}{u_m^-}^{p_1})^{1/{p_1}}\le C(\fint_{B_{2R}}{u_m^-}^p)^{1/p},
    \end{equation}
    for some $C=C(n,\lambda,\Lambda,\Gamma,p,p_1)$.\par

    This claim can be justified by considering $0<k<\frac{1}{2}$. Observe that when $k$ is in this range, $\vp\ge0$ in $B_{4R}$.
    Since $u\ep$ is a supersolution, we have
    $$
    k\int_{T_{4R}}A\nabla u\ep\nabla(\eta^2(u\ep^{2k-1}-(m+\epsilon)^{2k-1}))\ge0.
    $$

    $\Rightarrow$

    \begin{equation}
      \frac{k(2k-1)}{-2k}\int_{T_{2R}}a\nabla u\ep\cdot\nabla u\ep u\ep^{2k-2}\eta^2
      \le\int_{T_{2R}}(a+b)\nabla u\ep\cdot\nabla\eta(u\ep^{2k-1}-(M+\epsilon)^{2k-1})\eta.
    \end{equation}

    Note that $\supp\eta\subset B_r$ and that $\nabla u\ep=0$ in $B_r\setminus T_r$.
    We estimate
    \begin{align*}
    &\frac{2k-1}{-2k}\lambda\int_{B_r}\abs{\nabla v}^2\eta^2
    \le \frac{2k-1}{-2k}\int_{B_r}a\nabla v\cdot\nabla v\eta^2\\
    &=\frac{k(2k-1)}{2}\int_{B_{r}}a\nabla u\ep\cdot\nabla u\ep u\ep^{2k-2}\eta^2\\
    &\le\Lambda k\int_{B_r}\abs{\nabla u\ep}\abs{\nabla\eta}\abs{\eta}\abs{u\ep^{2k-1}}
    +k\int_{B_r}\abs{\tilde{b}-(\tilde b)_{B_r}}\abs{\nabla
    u\ep}\abs{\nabla\eta}\abs{u\ep^{2k-1}}\abs{\eta}\\
    &=\Lambda\int_{B_r}\abs{\nabla v}\abs{v}\abs{\nabla\eta}\abs{\eta}+
    \int_{B_r}\abs{\tilde{b}-(\tilde b)_{B_r}}\abs{\nabla v}\abs{v}\abs{\nabla\eta}\abs{\eta}\\
    &\le (\Lambda+\Gamma)r^{\frac{n}{s'}}(r-r')^{-1}(\int_{B_r}\abs{\nabla v}^2\abs{\eta}^2)^{1/2}
    (\int_{B_r}\abs{v}^q)^{1/q}.
    \end{align*}

   $\Rightarrow$
   $$
   \fint_{B_{r'}}\abs{\nabla v}^2\le
   C(n,\lambda,\Lambda,\Gamma,s,n)\theta^2(k)(r-r')^2(\fint_{B_r}\abs{v}^q)^{1/q},
   $$
   where $\theta(k)=\frac{2k}{2k-1}$.
   By Sobolev embedding we have
   $$
   (\fint_{B_{r'}}u\ep^{kql})^{\frac1{ql}}\le
   C(n,\lambda,\Lambda,\Gamma,s,n)\theta^2(k)(r-r')^2(\fint_{B_r}u\ep^{kq})^{1/q},
   $$
   where $l=\frac{2n}{(n-2)q}$.\par

    Let $k_i=kl^i$, $r=r_i=R+\frac{R}{2^i}$ and $r'=r_{i+1}$. Then we can iterate as long as $0<k_i<\frac1{2}$. Recall that $2<q<\frac{2n}{n-2}$. Then it is easy to see that
    for any $0<p<p_1<\frac{n}{n-2}$,
    \begin{equation}
      (\fint_{B_R}(u_m^-)^{p_1})^{1/{p_1}}\le C_{p,p_1}(\fint_{B_{2R}}(u_m^-)^p)^{1/p}.
    \end{equation}

    Claim 2. For any $p<0$,
    \begin{equation}
      \inf_{B_R}u_m^-\ge C(\fint_{B_{2R}}(u_m^-)^p)^{1/p}
    \end{equation}
    for some $C=C(n,\lambda,\Lambda,\Gamma,p)$.\par
    This time we consider $k<0$. We have
    $$
    k\int_{T_{4R}}A\nabla u\ep\nabla(\eta^2(u\ep^{2k-1}-(m+\epsilon)^{2k-1}))\le0.
    $$
    Then it is easy to get
    $$
    \frac{2k-1}{2k}\lambda\int_{B_r}\abs{\nabla v}^2\eta^2\le (\Lambda+\Gamma)r^{\frac{n}{s'}}(r-r')^{-1}(\int_{B_r}\abs{\nabla v}^2\eta^2)^{1/2}(\int_{B_r}\abs{v}^q)^{1/q}.
    $$

    Using Sobolev inequality and then iterations, and letting $\epsilon$ tend to 0, we obtain
    \begin{equation}
      \sup_{B_R}(u_m^-)^k\le C(n,\lambda,\Lambda,\Gamma,s)(\fint_{B_{2R}}{u_m^-}^{kq})^{1/q}.
    \end{equation}
    Since $k<0$, this implies
    \begin{equation}
      (\inf_{B_R}u_m^-)^k\le C(n,\lambda,\Lambda,\Gamma,s)(\fint_{B_{2R}}{u_m^-}^{kq})^{1/q}.
    \end{equation}
    Raise both sides to the power of $\frac1{k}$ we prove the claim.\par

    We will show
    \begin{equation}\label{BdyHolderLem2bd}
      (\fint_{B_{2R}}u\ep^{-\alpha})^{\frac1{-\alpha}}\ge C(\fint_{B_{2R}}u\ep^{\alpha})^{\frac1{\alpha}},
    \end{equation}
    where $\alpha$ depending on $n$ only is as in Lemma \ref{John-Niren}.
    If this is true, then by Claim 1 and Claim 2,
    $$
    \inf_{B_R}u_m^-\ge C(\fint_{B_{2R}}(u_m^-)^{-\alpha})^{\frac1{-\alpha}}
    \ge C(\fint_{B_{2R}}(u_m^-)^{\alpha})^{\frac1{\alpha}}\ge C_{p}(\fint_{B_R}{u_m^-}^p)^{\frac1{p}}
    $$
    for any $\alpha<p<\frac{n}{n-2}$. And trivially, for any $0<p_0\le p$
    $$
    (\fint_{B_R}{u_m^-}^p)^{\frac1{p}}\ge (\fint_{B_R}{u_m^-}^{p_0})^{\frac1{p_0}}.
    $$

    So it suffices to show \eqref{BdyHolderLem2bd}. As we did in the proof of Lemma \ref{Harnackineq}, let $w=\log u\ep$. We will see that $w\in BMO(B_{2R})$. In fact, fix any $B_r\subset B_{2R}$, let
    $$
    \phi=(u\ep^{-1}-(m+\epsilon)^{-1})\zeta^2,
    $$
    where $\zeta\in C_0^{\infty}(B_{4R})$, $\supp\zeta\subset B_{2r}$, $\zeta\equiv1$ in $B_r$ and $\abs{\nabla\zeta}\le C/r$. Then $\phi\ge 0$ in $T_{4R}$ and $\phi\in W_0^{1,2}(T_{4R})$. So
    $$
    \int_{T_{4R}}A\nabla u\ep\cdot\nabla\phi\ge0.
    $$
    Note that $\nabla u\ep\equiv0$ in $B_{2r}\setminus\Omega$, and  that
    $$
    0\le u\ep^{-1}-(m+\epsilon)^{-1}\le u\ep^{-1}.
    $$

    We get
    \begin{equation}
      \frac{\lambda}{2}\int_{B_{2r}}\abs{\nabla w}^2\zeta^2\le
      (\Lambda+C(n)\Gamma)(\int_{B_{2r}}\abs{\nabla w}^2\zeta^2)^{1/2}r^{\frac{n}{2}-1}.
    \end{equation}

     By the Poincar{\'e} inequality we get
     $$
     \fint_{B_r}\abs{w-(w)_{B_r}}^2\le r^2\fint_{B_r}\abs{\nabla w}^2\le C_1^2(n,\lambda,\Lambda,\Gamma).
     $$
    $\Rightarrow$
    $$
    w\in BMO(B_{2R}) \quad \text{with } \norm{w}_{BMO}\le C_1.
    $$

    Then by Lemma \ref{John-Niren} we proved \eqref{BdyHolderLem2bd}.

  \end{proof}

  Since the exterior Corkscrew condition gives 
  $\liminf\limits_{R\rightarrow0}\frac{\abs{B_R(P)\setminus\Omega}}{R^n}\ge (\frac1{2M})^n$,
  where $P\in\bdy\Omega$ and $M>1$ is the constant in Definition \ref{CorkscrewDefn},
  we can now prove regularity at the boundary using Lemma \ref{BdyHOlderLem2} and a standard argument. See for example \cite{gilbarg2001elliptic} Theorem 8.27. 

  \begin{lem}[boundary H{\"o}lder continuity]\label{BdyHolderCont}

  Let $u\in W^{1,2}(\Omega)$ be a solution in $\Omega$ and $P$ be a point on the boundary of $\Omega$. Let $B_R$ denote $B_R(P)$, $T_R$ denote $B_R(P)\cap\Omega$ and $\Delta_R$ denote $B_R(P)\cap\bdy\Omega$. Then for any $0<r\le R$, we have
  \begin{equation}\label{bdyHolder}
    \underset{T_{r}}{\osc} u\le C(\frac{r}{R})^{\alpha}\sup_{T_R}\abs{u}+\sigma(\sqrt{rR}),
  \end{equation}
  where $\underset{T_{r}}{\osc}u=\sup_{T_r}u-\inf_{T_r}u$, $\sigma(r)=\underset{\Delta_{r}}{\osc}u=\sup_{\Delta_r}u-\inf_{\Delta_r}u$, and $C=C(n,\Lambda,\lambda,\Gamma,M)$, $\alpha=\alpha(n,\Lambda,\lambda,\Gamma,M)$ are positive constants.\par
  In particular, if $u=0$ on $\Delta_R$, $0<r\le \frac{R}{2}$, then we have
  \begin{equation}\label{bdyHolderL2form}
    \underset{T_{r}}{\osc}u\le C(\frac{r}{R})^{\alpha}(\fint_{T_{R}}\abs{u}^2)^{1/2}.
  \end{equation}

  \end{lem}

  \begin{re}\label{ReCont}
    If, in addition to the assumptions of Lemma \ref{BdyHolderCont}, $\osc_{\Delta_R}u\rightarrow0$ as $R\rightarrow 0$, then \eqref{bdyHolder} implies that
    $$
    u(P_0)=\lim_{x\rightarrow P_0, x\in\Omega}u(x)
    $$
    is well-defined.
  \end{re}

  \begin{lem}[positivity]\label{positivity}
  Let $u\in W_{loc}^{1,2}(\Omega)\cap C(\overline{\Omega})$ be a supersolution of \eqref{Lu=0}. $u\ge0$ on $\bdy\Omega$. Then $u\ge0$ in $\Omega$.
  \end{lem}

  \begin{proof}
    For any $\epsilon>0$, set $v\ep=\min\{u,-\epsilon\}+\epsilon$. Then $v\ep\le0$ and
    $$
    \nabla v\ep=
    \begin{cases}
      \nabla u \quad u<-\epsilon,\\
      0 \qquad u\ge -\epsilon
    \end{cases}
    $$
    $\Rightarrow$ $v\ep\in W^{1,2}(\Omega)$ with $\supp v\ep\subset\subset\Omega$. Since $u$ is a supersolution, we have
    $$
      0\ge\int_{\Omega}A\nabla u\cdot \nabla v\ep=\int_{\{u<-\epsilon\}}A\nabla u\cdot\nabla u
      =\int_{\Omega}A\nabla v\ep\cdot\nabla v\ep
      \ge\lambda\int_{\Omega}\abs{\nabla v\ep}^2,
    $$
    $\Rightarrow$ $\nabla v\ep=0$ a.e. in $\Omega$, thus $u\ge-\epsilon$ in $\Omega$. Letting $\epsilon\rightarrow0$, we obtain $u\ge0$ in $\Omega$.

  \end{proof}

  This lemma immediately yields the maximum principle:
  \begin{lem}[maximum principle]\label{maxprinciple}
    Let $u\in W_{loc}^{1,2}(\Omega)\cap C(\overline{\Omega})$ be a subsolution of \eqref{Lu=0}. Then
    \begin{equation}\label{maxp1}
      \sup_{\Omega}u\le\sup_{\bdy\Omega}u.
    \end{equation}
    As a consequence, if $u\in W_{loc}^{1,2}(\Omega)\cap C(\overline{\Omega})$ is a weak solution of \eqref{Lu=0}, then
    \begin{equation}
      \sup_{\Omega}\abs{u}\le\sup_{\bdy\Omega}\abs{u}.
    \end{equation}
  \end{lem}

  \begin{proof}
    Let $M=\sup_{\bdy\Omega}u$. We can assume $M<\infty$ since otherwise \eqref{maxp1} is trivial. Then apply Lemma \ref{positivity} to $M-u$ to get $M-u\ge0$ in $\Omega$.
  \end{proof}

  \begin{lem}\label{Lipbdy}
    Let $\Omega$ be an NTA domain, and suppose $g\in Lip(\bdy\Omega)$. Then there exists a unique $u\in W^{1,2}(\Omega)$ that solves the classical Dirichlet problem with data $g$. Moreover, $u\in C^{0,\beta}(\overline\Omega)$ for some $0<\beta<1$.
  \end{lem}

  \begin{proof}
    Let $G\in Lip_0(\Rn)$ be such that $G|_{\bdy\Omega}=g$. Then by Theorem \ref{LaxMilgramLem}, there exists a unique solution $u\in W^{1,2}(\Omega)$ to the classical Dirichlet problem
    $$
    \begin{cases}
      Lu=0 \qquad \text{in } \Omega,\\
      u-G\in W_0^{1,2}(\Omega).
    \end{cases}
    $$

    By Remark \ref{ReCont},
    $$
    \lim_{\Omega\ni X\rightarrow P}u(X)=g(P) \quad \forall P\in \bdy\Omega.
    $$

    To see that $u\in C^{0,\beta}(\overline\Omega)$, fix $P_0\in\bdy\Omega$ and $X\in\Omega$ with $r\doteq\abs{X-P_0}<\frac1{2}$.


    By Lemma \ref{BdyHolderCont},
    \begin{align*}
      \abs{u(X)-g(P_0)}&\le \underset{T_{2r}(P_0)}{\osc}u\le Cr^{\alpha}\sup_{T_1}\abs{u}+\underset{\Delta_{\sqrt{2r}}(P_0)}{\osc} g\\
      &\le C r^{\alpha}\norm{g}_{C^{0,1}(\bdy\Omega)}+\sqrt{2}[g]_{C^{0,1}(\bdy\Omega)}r^{1/2}\\
      &\le C\norm{g}_{C^{0,1}(\bdy\Omega)}\abs{X-P_0}^{\alpha'}
    \end{align*}
    for some $0<\alpha'<1$.

    So together with Lemma \ref{intHolderLem}, we can show that for any $X,Y\in\Omega$,
    $$
    \frac{\abs{u(X)-u(Y)}}{\abs{X-Y}^{\beta}}\le C \quad\text{ for some } \beta\in(0,1).
    $$

  \end{proof}

  Now we can immediately obtain the following

  \begin{thm}\label{contDirichletSol}
    Let $\Omega$ be an NTA domain. Then the continuous Dirichlet problem \eqref{contDirichlet} is uniquely solvable.
  \end{thm}

  \begin{proof}
    Let $f\in C(\bdy\Omega)$. Then there exists a sequence $\{f_n\}_{n=1}^{\infty}\subset Lip(\bdy\Omega)$ such that $f_n\rightarrow f$ uniformly on $\bdy\Omega$. By Lemma \ref{Lipbdy}, the corresponding solutions $u_n\in W^{1,2}(\Omega)\cap C^{0,\beta}(\overline{\Omega})$. By the maximum principle,
    $$
    \sup_{\Omega}\abs{u_i-u_j}\le \sup_{\bdy\Omega}\abs{f_i-f_j}.
    $$

    So $u_n$ converges uniformly to $u\in C(\overline\Omega)$. By Caccioppoli inequality (Corollary \ref{caccio}), $u\in W_{loc}^{1,2}(\Omega)$, and also $Lu=0$ in $\Omega$, $u=g$ on $\bdy\Omega$. If $v\in W_{loc}^{1,2}(\Omega)\cap C(\overline{\Omega})$ is another solution, then $u-v\in W_0^{1,2}(\Omega)$ and by the maximum principle, $u\equiv v$.
  \end{proof}

\section{Green's functions and regular points}\label{GreenSection}
In this section, we collect some results about the Green's function and regular points. It turns out that the arguments in \cite{gruter1982green} carry forward with only a few modifications. Thus, we will mainly focus on these modifications.\par
Recall that $B[u,v]=\int_{\Omega}(a\nabla u\cdot\nabla v+b\nabla u\cdot \nabla v)$.

\begin{thm}
  There exists a unique function (the Green's function) $G:\Omega\times\Omega \rightarrow \Real\cup\{\infty\}$, such that for each $Y\in\Omega$ and any $r>0$
  \begin{equation}\label{GW1.3}
    G(\cdot,Y)\in W^{1,2}(\Omega\setminus B_r(Y))\cap W_0^{1,1}(\Omega)
  \end{equation}

  and for all $\phi\in W_0^{1,p}(\Omega)\cap C(\Omega)$ where $p>n$
  \begin{equation}\label{GW1.4}
    B[G(\cdot,Y),\phi]=\phi(Y).
  \end{equation}

  The Green's function enjoys the following properties: for each $Y\in\Omega$ ($G(X)\doteq G(X,Y)$)
  \begin{equation}\label{GW1.5}
    G\in L^{\frac{n}{n-2},\infty}(\Omega) \quad\text{with } \norm{G}_{L^{\frac{n}{n-2},\infty}}\le C(n)\lambda^{-1},
  \end{equation}

  \begin{equation}\label{GW1.6}
    \nabla G\in L^{\frac{n}{n-1},\infty}(\Omega) \quad\text{with } \norm{\nabla G}_{L^{\frac{n}{n-1},\infty}}\le C(n,\Lambda,\lambda,\Gamma),
  \end{equation}

  \begin{equation}\label{GW1.7}
    G\in W_0^{1,k}(\Omega) \quad\text{for each } k\in[1,\frac{n}{n-1}),\quad\text{and }\norm{G}_{W_0^{1,k}(\Omega)}\le C(n,\lambda,\Lambda,\Gamma,k).
  \end{equation}

  For all $X,Y\in\Omega$ we have
  \begin{equation}\label{GW1.8}
    G(X,Y)\le C(n,\Lambda,\lambda,\Gamma)\abs{X-Y}^{2-n};
  \end{equation}

  and for all $X,Y\in\Omega$ satisfying $\abs{X-Y}\le \frac1{2}\dist(Y,\bdy\Omega)$
  \begin{equation}\label{GW1.9}
    G(X,Y)\ge C(n,\Lambda,\lambda,\Gamma)\abs{X-Y}^{2-n}.
  \end{equation}
\end{thm}

\begin{proof}
  Let $Y\in\Omega$ be fixed. For fixed $\rho>0$, ($B_{\rho}=B_{\rho}(Y)$),
  $$
  W_0^{1,2}(\Omega)\ni \phi\mapsto \fint_{B_{\rho}}\phi
  $$
  is a bounded linear functional on $W_0^{1,2}(\Omega)$. By \eqref{energyest}, \eqref{ellipticity} and the Lax-Milgram theorem, there exists a unique function $G^{\rho}\in W_0^{1,2}(\Omega)$, such that for all $\phi\in W_0^{1,2}(\Omega)$
  \begin{equation}\label{GW1.14}
    B[G^{\rho},\phi]=\fint_{B_{\rho}}\phi.
  \end{equation}

  Using the same argument as in \cite{gruter1982green}, we obtain $G^\rho\ge0$, and
  \begin{equation}
  \norm{G^{\rho}}_{L^{\frac{n}{n-2},\infty}}\le C(n)\lambda^{-1},
  \end{equation}
  keeping in mind that
  $$
  \int_{\Omega_t}b\nabla G^{\rho}\nabla G^{\rho}(G^{\rho})^{-2}=0,
  $$
  where $\Omega_t=\{X\in\Omega: G^\rho(X)>t\}$.

  We also have a pointwise estimate for $G^{\rho}$:
  \begin{equation}\label{GW1.28}
    G^{\rho}(X)\le C(n,\lambda,\Lambda,\Gamma)\lambda^{-1}\abs{X-Y}^{2-n} \quad\text{if } \abs{X-Y}\ge2\rho.
  \end{equation}
  Here the only difference is that the constant depends also on $\Lambda$ and $\Gamma$, since the constant in Corollary \ref{MoserCor1} has such dependence. \par

  Now we show
  \begin{equation}\label{GW1.31}
    \norm{\nabla G^{\rho}}_{L^{\frac{n}{n-1},\infty}(\Omega)}\le C(n,\lambda,\Lambda,\Gamma),
  \end{equation}
and it suffices to establish that
  \begin{equation}
    \int_{\Omega\setminus B_R}\abs{\nabla G^{\rho}}^2\le C(n,\lambda,\Lambda,\Gamma)R^{2-n} \quad\text{if } R\ge4\rho.
  \end{equation}

  Let $\eta\in C^{\infty}$ satisfying $\eta\equiv1$ outside of $B_R$, $\eta\equiv0$ in $B_{R/2}$ and $\abs{\nabla\eta}\le C/R$, and take $\phi = G^{\rho}\eta^2$ in \eqref{GW1.14}. Then, as in the proof of Corollary \ref{caccio'}, and using \eqref{GW1.28}, we have
  \begin{align*}
    \lambda\int_{\Omega}\abs{\nabla G^{\rho}}^2\eta^2&\le C(n,\lambda,\Lambda,s)(1+\Gamma^2)R^{2(\frac{n}{s'}-1)}\Big(\int_{B_R\setminus B_{R/2}}\abs{G^{\rho}}^{\frac{2s}{2-s}}\Big)^{\frac{2-s}{s}}\\
    &\le C(n,\lambda,\Lambda,\Gamma)R^{2-n}.
  \end{align*}

  We now consider the convergence of $G^\rho$.
  By \eqref{GW1.31}, we have for each $k\in[1,\frac{n}{n-1})$ a uniform bound on $\norm{G^{\rho}}_{W_0^{1,k}}$ with respect to $\rho$. Thus there exists a $G\in W_0^{1,k}$ for all $k\in[1,\frac{n}{n-1})$ such that
  \begin{equation}
    G^{\rho_{\mu}}\rightharpoondown G \quad\text{in } W_0^{1,k},
  \end{equation}
  and $\norm{G}_{W_0^{1,k}}\le\norm{G^{\rho_{\mu}}}_{W_0^{1,k}}\le C(n,\lambda,\Lambda,\Gamma,k)$, which is \eqref{GW1.7}.

  Let $\phi\in W_0^{1,p}(\Omega)\cap C(\Omega)$ with $p>n$. We verify that $B[\cdot,\phi]$ is a continuous linear functional on $W_0^{1,p'}$ where $p'=\frac{p}{p-1}<\frac{n}{n-1}$. Since
  \begin{equation}\label{bound1}
  \abs{\int_{\Omega}a\nabla u\cdot\nabla\phi}\le \Lambda(\int_{\Omega}\abs{\nabla u}^{p'})^{1/{p'}}(\int_{\Omega}\abs{\nabla \phi}^p)^{1/p},
  \end{equation}
  we use the BMO-extension $\tilde{b}_{ij}$ of $b_{ij}$ (\eqref{BMOextension})
   and the zero extensions $\tilde{u}$ and $\tilde{\phi}$ of $u$ and $\phi$ to $\Rn$ to see that

  \begin{equation}\label{bound2}
    \abs{\int_{\Omega}b\nabla u\cdot\nabla\phi}
      \le C\norm{b}_{BMO(\Omega)}\norm{u}_{W^{1,p'}(\Omega)}\norm{\phi}_{W^{1,p}(\Omega)}. \end{equation}

  Combining \eqref{bound1} and \eqref{bound2} we obtain
  \begin{equation}
    \abs{B[u,\phi]}\le C\norm{u}_{W^{1,p'}(\Omega)}\norm{\phi}_{W^{1,p}(\Omega)} \quad\forall u\in W_0^{1,p'},
  \end{equation}
  which shows that $B[\cdot,\phi]\in(W_0^{1,p'})^*$.

  Since $G^{\rho}\rightharpoondown G$ in $W_0^{1,k}$ for all $k\in[1,\frac{n}{n-1})$,
  \begin{equation}\label{GrhoG}
  B[G^{\rho},\phi]\rightarrow B[G,\phi].
  \end{equation}

  On the other hand,
  $$
  B[G^{\rho},\phi]=\fint_{B_{\rho}}\phi\rightarrow \phi(y),
  $$
  so $B[G,\phi]=\phi(y)$, which is \eqref{GW1.4}. From here it is not hard to obtain \eqref{GW1.7}. And \eqref{GW1.8} follows from the pointwise estimate \eqref{GW1.28} and H\"older continuity of $G(\cdot,y)$ in $\Omega\setminus\{y\}$.

  We now give the proof of \eqref{GW1.9}:\par
  Let $X,Y\in\Omega$, $r=\abs{X-Y}<\frac1{2}\dist(Y,\bdy\Omega)$. Let $\eta\in C_0^{\infty}(\Omega)$ with $\supp\eta\subset Q_{3r/2}(Y)\setminus Q_{r/4}(Y)$ and $\eta\equiv1$ in $Q_r(Y)\setminus Q_{r/2}(Y)$. Then replacing the test function in the proof of Corollary \ref{caccio'} by $\eta$, we get
  \begin{equation}\label{GW1.9pf1}
    \int_{Q_r\setminus Q_{r/2}}\abs{\nabla G}^2\le C(n,\Lambda,\lambda,\Gamma,s)r^{2(\frac{n}{s'}-1)}
   \Big(\int_{Q_{3r/2}\setminus Q_{r/4}}\abs{G}^{\frac{2s}{2-s}}\Big)^{\frac{2-s}{s}}.
  \end{equation}

  Now consider a cut-off function $\vp\in C_0^{\infty}$ with $\supp\vp\subset Q_r(Y)$, $\vp\equiv1$ on $Q_{r/2}(Y)$ and $\abs{\nabla\vp}\le C/r$. Then inserting $\vp$ in \eqref{GW1.4}, we have
  \begin{equation}\label{GW1.9pf0}
    1=\int_{Q_r}a\nabla G\cdot\nabla\vp+\int_{Q_r}b\nabla G\cdot\nabla\vp.
  \end{equation}

  We claim that
  \begin{equation}\label{GW1.9pf2}
    \int_{Q_r}b\nabla G\cdot\nabla\vp\le C(n,\Lambda,\lambda,\Gamma,s)\abs{X-Y}^{n-2}G(X,Y).
  \end{equation}

  Actually,
  \begin{align*}
    \int_{Q_r}b\nabla G\cdot\nabla\vp&=\int_{Q_r}(b-(b)_{Q_r})\nabla G\cdot\nabla\vp\\
    &\le Cr^{\frac{n}{2}-1}(\fint_{Q_r}(b-(b)_{Q_r})^2)^{1/2}(\int_{Q_r\setminus Q_{r/2}}\abs{\nabla G}^2)^{1/2}\\
    &\le C(n,\Lambda,\lambda,\Gamma,s)r^{\frac{n}{2}+\frac{n}{s'}-2}\Big(\int_{Q_{3r/2}\setminus Q_{r/4}}\abs{G}^{\frac{2s}{2-s}}\Big)^{\frac{2-s}{2s}}\\
    &\le C(n,\Lambda,\lambda,\Gamma)r^{n-2}\sup_{Q_{3r/2}\setminus Q_{r/4}}G\\
    &\le C(n,\Lambda,\lambda,\Gamma)r^{n-2}\inf_{Q_{3r/2}\setminus Q_{r/4}}G\\
    &\le C(n,\Lambda,\lambda,\Gamma)\abs{x-y}^{n-2}G(X,Y),
  \end{align*}
  where we used \eqref{GW1.9pf1} to obtain the second inequality and Harnack inequality in the fourth inequality.\par
  Similarly,
  $$
  \int_{Q_r}a\nabla G\cdot\nabla\vp\le \Lambda r^{\frac{n}{2}-1}(\int_{Q_r\setminus Q_{r/2}}\abs{\nabla G}^2)^{1/2}\le C(n,\Lambda,\lambda,\Gamma)\abs{X-Y}^{n-2}G(X,Y).
  $$

  Now by \eqref{GW1.9pf0},
  $$
  1\le C(n,\Lambda,\lambda,\Gamma)\abs{X-Y}^{n-2}G(X,Y),
  $$
  which gives \eqref{GW1.9}.\par

  Proof of uniqueness:\par
  We only give the proof of (1.49) in \cite{gruter1982green}, which is the only place something different occurs. But it again follows from a variation of Corollary \ref{caccio'} and Harnack inequality:
  \begin{align*}
    \int_{\Omega\setminus B_{2\rho_{\nu}}(Y)}\abs{\nabla u}^2&\le C(n,\Lambda,\lambda,\Gamma)\rho_{\nu}^{2(\frac{n}{s'}-1)}\Big(\int_{B_{2\rho_{\nu}}\setminus B{\rho_{\nu}}}\abs{u}^{\frac{2s}{2-s}}\Big)^{\frac{2-s}{s}}\\
    &\le C(n,\Lambda,\lambda,\Gamma)\rho_{\nu}^{n-2}(\sup_{B_{2\rho_{\nu}}\setminus B{\rho_{\nu}}}\abs{u})^2\\
    &\le C(n,\Lambda,\lambda,\Gamma)\rho_{\nu}^{n-2}m^*(\rho_{\nu})^2,
  \end{align*}
where $m^*(\rho)\doteq\inf_{\bdy B_{\rho}(Y)}u$.

\end{proof}

The standard relations hold between Green's functions of the operator and its adjoint:

\begin{thm}
  Let $L^*=-\divg A^*\nabla$ be the adjoint operator to $L$ and consider the Green function $G$ and $G^*$ corresponding to $L$ and $L^*$. Then for all points $X,Y\in\Omega$, we have
  \begin{equation}
    G(X,Y)=G^*(Y,X).
  \end{equation}
\end{thm}

And the following representation formula.
\begin{prop}
  For any $X,Y\in\Omega$, $\rho>0$ such that $\rho<\dist(Y,\bdy\Omega)$
  \begin{equation}
    G^{\rho}(X,Y)=\fint_{B_{\rho}(Y)}G(X,Z)dZ.
  \end{equation}
\end{prop}

In the rest of this section, we apply the method of \cite{gruter1982green} to investigate the regular points for operator $L=-\divg A\nabla$ with $A$ satisfying \eqref{aBound}--\eqref{bBMO}.

\begin{defn}
  A point $Q\in\bdy\Omega$ is said to be regular for $L$, if for any Lipschitz function $h$ on $\bdy\Omega$, the $W^{1,2}$ solution $u$ to $Lu=0$ in $\Omega$ with boundary data $h$ satisfies
  \begin{equation}
    \lim_{X\rightarrow Q, X\in\Omega}u(X)=h(Q).
  \end{equation}
  If every point $Q\in\bdy\Omega$ is a regular point, then we say the domain $\Omega$ is regular.
\end{defn}

The main result is the following.
\begin{thm}\label{GWthm2.4}
  A point of $\bdy\Omega$ is a regular point for $L=-\divg A\nabla$ with $A$ satisfying \eqref{aBound}--\eqref{bBMO} if and only if it is a regular point for the Laplacian.
\end{thm}

Let us first define $\capc_{L}$ of a compact set.

Let $E\subset\Omega$ be compact and denote by $K_E$ the following closed convex subset of $W_0^{1,2}(\Omega)$
\begin{equation}
  K_E\doteq\{v\in W_0^{1,2}(\Omega):v\ge1 \text{ on }E \text{ in the sense of }W^{1,2}\}.
\end{equation}
Then there exists a unique solution $u\in K_E$ of the variational inequality
\begin{equation}\label{GW2.3}
  B[u,v-u]\ge0 \quad\forall v\in K_E.
\end{equation}

The existence of $u$ can be justified by Theorem 2.1 in \cite{kinderlehrer1980introduction}, which only requires that the bilinear form $B[\cdot,\cdot]$ be coercive. \par
This function $u$ is called the equilibrium potential of $E$. Letting $v=\{u\}^1$, the truncation of $u$ at height 1 in \eqref{GW2.3}, we get $u\le 1$ a.e. in $\Omega$. Therefore,
\begin{equation}\label{u1onE}
u\equiv1 \quad\text{on }E\quad\text{in the sense of }W^{1,2}.
\end{equation}
From \eqref{GW2.3} we get for all $\phi\in C_0^{\infty}$ with $\phi\ge0$ on $E$
\begin{equation}\label{GW2.4}
  B[u,\phi]\ge0.
\end{equation}
Then we immediately get
\begin{prop}
  The equilibrium potential $u$ of the compact subset $E$ of $\Omega$ is the solution to
   $$
   \begin{cases}
     Lu=0\qquad\text{in }\Omega\setminus E,\\
     u|_{\bdy\Omega}=0,\quad u|_{\bdy E}=1\qquad\text{in the sense of }W_0^{1,2}.
   \end{cases}
   $$
   \end{prop}
   By \eqref{GW2.4} and the Riesz representation theorem, there exists a regular Borel measure $\mu$ with $\supp\mu\subset E$ such that
   \begin{equation}\label{GW2.5}
     B[u,\phi]=\int_{E}\phi d\mu\qquad\forall\, \phi\in C_0^{\infty}(\Omega).
   \end{equation}
   \eqref{u1onE} implies that $\supp\mu\subset\bdy E$. The measure $\mu$ is called the equilibrium measure of $E$.
\begin{defn}
  The capacity of $E$ with respect to the operator $L$ is defined as
  \begin{equation}
    \capc_{L}(E)\doteq\mu(E)=B[u,u].
  \end{equation}
\end{defn}

As in the classical elliptic setting, the capacity with respect to $L$ can be compared to the capacity with respect to the Laplacian. Precisely,
\begin{equation}\label{L&Laplace}
  \lambda\capc_{\Delta}(E)\le\capc_{L}(E)\le \lambda^{-1}C(n,\Lambda,\Gamma)\capc_{\Delta}(E),
\end{equation}
where $\capc_{\Delta}(E)=\inf\{\int_{\Omega}\abs{\nabla u}^2: u\in C_0^{\infty}(\Rn), u\ge1 \text{ on }E\}$.

To see this, simply observe that
\begin{align*}
  \capc_{\Delta}(E)&\le\int_{\Omega}\abs{\nabla u_L}^2\le \frac1{\lambda}\int_{\Omega}A\nabla u_L\cdot\nabla u_L
  =\frac1{\lambda}\capc_L(E),
\end{align*}
where we use $u_L$ to denote the equilibrium potential with respect to operator $L$ of set $E$.
On the other hand, we have
\begin{align*}
  \capc_L(E)&\leq B[u_L,u_{\Delta}]\le C(n,\Lambda,\Gamma)\norm{\nabla u_L}_{L^2(\Omega)}\norm{\nabla u_{\Delta}}_{L^2(\Omega)}\\
  &\le C(n,\Lambda,\Gamma)\Big(\frac1{\lambda}\int_{\Omega}A\nabla u_L\cdot\nabla u_L\Big)^{1/2}\capc_{\Delta}(E)^{1/2}\\
  &=C(n,\Lambda,\Gamma)\lambda^{-\frac1{2}}\capc_L(E)^{1/2}\capc_{\Delta}(E)^{1/2}.
  \end{align*}

 The equilibrium potential $u$ with respect to $L$ of $E\subset\Omega$ can be represented as follows
 \begin{lem}Let $\mu$ be the equilibrium measure of $E$. Let G be the Green function on $\Omega$. Then
 \begin{equation}\label{GW2.8}
   u(X)=\int_EG(X,Y)d\mu(Y).
 \end{equation}
 \end{lem}
 This can be proved by taking $\phi=G^{\rho}(X,\cdot)$ in \eqref{GW2.5}, and then using Fatou's lemma. One direction of the inequality using Fatou's lemma is immediate. For the other direction, note that we can show for $Y\in\Omega_{\delta}$, $X\in\Omega$, $G^{\rho}(X,Y)\le C(n,\lambda,\Lambda,\Gamma)(G(X,Y)+1)$, where $\Omega_{\delta}=\{X\in\Omega:\dist(X,\bdy\Omega)>\delta\}$ and $\delta=\dist(E,\bdy\Omega)/10$.

 After these preparations, the classical arguments go through and we obtain the following.

 \begin{lem} Fix $Q\in\bdy\Omega$ (we may assume $Q=O$). For $r>0$, set $C_r\doteq\Omega^{\complement}\cap B_r(O)$. Consider the weak solution $u\in W^{1,2}(\Omega)$ of $Lu=0$ in $\Omega$, $0\le u\le1$ on $\bdy\Omega$ and $u=0$ on $\bdy\Omega\cap B_{\rho}$. There exist constants $0<\alpha_0(n,\lambda,\Lambda,\Gamma)<\frac1{2}$ and $C(n,\lambda,\Lambda,\Gamma)>0$ such that for all $\alpha\le\alpha_0$ and $r<\alpha\rho$ we have
 \begin{equation}
   \sup_{\Omega\cap B_r}u\le
   \exp\Big\{\frac{C}{\log\alpha}\int_{r}^{\alpha\rho}\frac{\capc_L(C_t)}{t^{n-1}}dt\Big\}.
 \end{equation}
 \end{lem}

 \begin{thm}\label{GWthm2.5}
   The point $Q\in\bdy\Omega$ (we may assume $Q=O$) is a regular point for $Lu=0$ if and only if
   \begin{equation}
     \int_{0}\frac{\capc_L(C_r)}{r^{n-1}}dr=\infty,
   \end{equation}
   where $C_r=\Omega^{\complement}\cap B_r(O)$.
 \end{thm}

Then Theorem \ref{GWthm2.4} is a consequence of Theorem \ref{GWthm2.5} and \eqref{L&Laplace}.

  \section{Elliptic measure}\label{ellipticSec}

  Let $g\in C(\bdy\Omega)$, $X\in\Omega$. We showed in Section \ref{BdySection} that there exists a unique $u\in W_{loc}^{1,2}(\Omega)\cap C(\overline{\Omega})$ such that $Lu=0$ in $\Omega$ and $u=g$ on $\bdy\Omega$. Consider the linear functional on $C(\bdy\Omega)$
  $$
  T: g\mapsto u(X).
  $$
  By Lemma \ref{positivity} and Lemma \ref{maxprinciple}, $T$ is a positive and bounded linear functional. Moreover, if $g\equiv1$, $u\equiv1$. By the Riesz representation theorem, there exists a family of regular Borel probability measures $\{\omega_L^X\}_{X\in\Omega}$ such that
  \begin{equation}
    u(X)=\int_{\bdy\Omega}g(P)d\omega_L^X(P).
  \end{equation}
  This family of measures is called the elliptic measure associated to $L$, or $L$-harmonic measure. We omit the reference to $L$ when no confusion arises. For a fixed $X_0\in\Omega$, let $\omega=\omega^{X_0}$.\par
  We list some properties of the elliptic measure. Since the main tools used here are Harnack principle and H\"older continuity of the solution, most of the properties can be proved as in \cite{kenig1994harmonic}. We only provide a proof when needed.\par
  In what follows $C$ is a constant depending on $n,\lambda,\Lambda,\Gamma$ and the constants $M,m,m'$ in Definition \ref{CorkscrewDefn} and \ref{HarnackChainDefn} unless otherwise stated. Its value may vary from line to line.

  \begin{prop}
    $\forall \, X_1,X_2\in\Omega$, $\omega^{X_1}$ and $\omega^{X_2}$ are mutually absolutely continuous.
  \end{prop}

  Recall that for $Q\in\bdy\Omega$, $A_r(Q)$ is the corkscrew point in $\Omega$ such that $\abs{A_r(Q)-Q}<r$ and $\delta(A_r(Q))>\frac{r}{M}$.

  \begin{prop}
    Let $Q\in\bdy\Omega$. Then
    $\omega^{A_r(Q)}(\Delta_r(Q))\ge C$.
  \end{prop}

  \begin{prop}
    Suppose that $u\ge0$, $Lu=0$, $u\in W^{1,2}(T_{2r}(Q))\cap C(\overline{T_{2r}(Q)})$, $u\equiv0$ on $\Delta_{2r}(Q)$. Then
    $u(X)\le Cu(A_r(Q))$  $\forall\, X\in T_{3r/2}(Q)$.
  \end{prop}
  For a proof see \cite{jerison1982boundary} Lemma 4.4.

  \begin{prop}
  $r^{n-2}G(X,A_r(Q))\le C\omega^X(\Delta_{2r}(Q))$ $\forall\, X\in\Omega\setminus B_{r/2}(A_r(Q))$,
  where $G$ is the Green's function defined in Section \ref{GreenSection}.
  \end{prop}

  \begin{prop}
      $\omega^X(\Delta_r(Q))\le C(n,\Lambda,\Gamma)r^{n-2}G(X,A_r(Q))$  $\forall\, X\in\Omega\setminus B_{2r}(Q)$.
  \end{prop}

  \begin{proof}
    Let $\vp\in C_0^{\infty}(\Rn)$, $u$ be the weak solution of
    $$
    \begin{cases}
      Lu=0 \quad \text{in }\Omega,\\
      u=\vp \quad \text{ on }\bdy\Omega.
    \end{cases}
    $$
    Then as we showed in Section \ref{BdySection}, $u\in W^{1,2}(\Omega)\cap C(\overline{\Omega})$, $u-\vp\in W_0^{1,2}(\Omega)$.\\
    \textbf{Claim.} For any $Y\notin\overline{\supp\vp}$,
    \begin{equation}\label{claim1}
      u(Y)=-\int_{\Omega}(a+b)\nabla\vp(X)\cdot\nabla G^*(X,Y)dX,
    \end{equation}
    where $G^*(X,Y)$ is the Green function of the adjoint operator $L^*=\divg A^T\nabla$.\par

    To justify this claim, consider $(G^*)^\rho\in W_0^{1,2}(\Omega)$. We have (see \eqref{GW1.14})
    \begin{equation}
    B[\phi, (G^*)^\rho]=\fint_{B_\rho(Y)}\phi \qquad\forall\, \phi\in W_0^{1,2}(\Omega),
    \end{equation}
    and (see \eqref{GrhoG})
    \begin{equation}
    B[\phi, (G^*)^\rho]\rightarrow B[\phi,G^*] \text{ as } \rho\rightarrow0, \quad\forall\, \phi\in W_0^{1,p}(\Omega)\cap C(\Omega) \text{ with }p>n.
    \end{equation}

    So
    \begin{equation}\label{claim2}
      \fint_{B_\rho(Y)}(u-\vp)=\int_{\Omega}A\nabla(u-\vp)\cdot\nabla (G^*)^{\rho}.
    \end{equation}

    On the other hand, since $u$ is a solution,
    $$
    \int_{\Omega}A\nabla u\cdot\nabla (G^*)^\rho=0.
    $$

    Then the right-hand side of \eqref{claim2} equals
    $-\int_{\Omega}A\nabla\vp\cdot\nabla(G^*)^{\rho}$,
    which goes to $-\int_{\Omega}A\nabla\vp\cdot\nabla G^*$
    as $\rho\rightarrow0$,
    while
    the left-hand side of \eqref{claim2} goes to $u(Y)$. Thus we obtain \eqref{claim1}.\par

    Now pick $\vp\equiv1$ in $B_r(Q)$, with $\supp\vp\subset B_{3r/2}(Q)$, $0\le\vp\le1$ and $\abs{\nabla\vp}\le\frac{C}{r}$. Then by maximum principle and the claim,
    \begin{equation}
      \omega^Y(\Delta_r(Q))\le u(Y)=-\int_{\Omega}(a+b)\nabla\vp(X)\cdot\nabla G^*(X,Y)dX.
    \end{equation}

    We estimate
    \begin{align}\label{ellmeas_a}
      \abs{\int_{\Omega}a\nabla\vp\cdot\nabla G^*}&\le C(n,\Lambda)r^{n-1}\Big(\fint_{T_{3r}{2}(Q)}\abs{\nabla G^*}^2\Big)^{1/2}
      \le Cr^{n-2}\Big(\fint_{T_{7r}{4}(Q)}\abs{G^*}^2\Big)^{1/2}\nonumber\\
      &\le Cr^{n-2}G^*(A_r(Q),Y)=Cr^{n-2}G(Y,A_r(Q)).
    \end{align}

     Choose $\eta\in C_0^{\infty}(\Rn)$ with $\eta\equiv1$ on $B_{3r/2}(Q)$, $\supp\eta\subset B_{\frac{7r}{4}}(Q)$, and $\abs{\nabla\eta}\le\frac{C}{r}$.
    Then using the properties of $\eta$ and $\vp$, we have
    $$
    \int_{\Omega}b\nabla\vp\cdot\nabla G^*=\int_{\Rn}\tilde{b}\nabla\vp\cdot\nabla(G^*\eta),
    $$
    where $\tilde{b}\in BMO(\Rn)$ is the extension of $b$, and we identify $G^*(\cdot,Y)$ with its zero extension outside of $\Omega$.

    Note that since $Y\notin B_{2r}(Q)$, $G^*(\cdot,Y)\in W^{1,2}(T_{3r/2}(Q))$. So we can apply Proposition \ref{compensatedProp} and get
    \begin{align}\label{ellmeas_b}
      \abs{\int_{\Omega}b\nabla\vp\cdot\nabla G^*}&\le C\Gamma\norm{\nabla\vp}_{L^2(\Rn)}\norm{\nabla (G^*\eta)}_{L^2(\Rn)}\nonumber\\
      &\le C\Gamma r^{\frac{n}{2}-1}(\norm{\eta\nabla G^*}_{L^2(\Rn)}+\norm{G^*\nabla\eta}_{L^2(\Rn)})\nonumber\\
      &\le C\Gamma r^{\frac{n}{2}-1}\Big\{\Big(\int_{T_{7r/4}(Q)}\abs{\nabla G^*}^2\Big)^{\frac1{2}}+\frac{C}{r}\Big(\int_{T_{7r/4}(Q)}\abs{G^*}^2\Big)^{\frac1{2}}\Big\}\nonumber\\
      &\le C\Gamma r^{n-2}\Big(\fint_{T_{2r}(Q)}\abs{G^*}^2\Big)^{1/2}\le C\Gamma r^{n-2}G(Y,A_r(Q)),
    \end{align}
    where $C$ depending on dimension and $\Omega$.

    Combining \eqref{ellmeas_a} and \eqref{ellmeas_b} we have proved $\omega^Y(\Delta_r(Q))\le Cr^{n-2}G(Y,A_r(Q))$, for any $Y\in\Omega\setminus B_{2r}(Q)$, where the constant $C$ depending on $n,\Lambda,\Gamma$ and the domain.

  \end{proof}

  \begin{cor}\label{omegaG}
    For $X\in\Omega\setminus B_{2r}(Q)$, $\omega^X(\Delta_r(Q))\approx r^{n-2}G(X,A_r(Q))$.
  \end{cor}

  \begin{cor}
    For $X\in\Omega\setminus B_{2r}(Q)$, $\omega^X$ is a doubling measure, i.e. $\omega^X(\Delta_{2r}(Q))\le C\omega^X(\Delta_r(Q))$.
  \end{cor}

  \begin{prop}[Comparison principle]\label{comparisonpp} Let $u,v\ge0$, $Lu=Lv=0$, $u,v\in W^{1,2}(T_{2r}(Q))\cap C(\overline{T_{2r}(Q)})$, $u,v\equiv0$ on $\Delta_{2r}(Q)$. Then $\forall\, X\in T_r(Q)$,
  $$
  C^{-1}\frac{u(A_r(Q))}{v(A_r(Q))}\le \frac{u(X)}{v(X)}\le C\frac{u(A_r(Q))}{v(A_r(Q))}
  $$
  \end{prop}

  \begin{cor}
    Let $u,v$ be as in Proposition \ref{comparisonpp}. Then there exists $\alpha=\alpha(n,\lambda,\Lambda,\Gamma)$ such that
    $$
    \abs{\frac{u(Y)}{v(Y)}-\frac{u(X)}{v(X)}}\le C\frac{u(A_r(Q))}{v(A_r(Q))}\Big(\frac{\abs{X-Y}}{r}\Big)^{\alpha} \qquad\forall\, X,Y\in T_r(Q).
    $$
  \end{cor}

  \begin{cor}
    Let $\Delta=\Delta_r(Q)$, $\Delta'=\Delta_s(Q_0)\subset\Delta_{r/2}(Q)$, $X\in B\setminus T_{2r}(Q)$. Then
    $$
    \omega^{A_r(Q)}(\Delta')\approx\frac{\omega^X(\Delta')}{\omega^X(\Delta)}.
    $$
  \end{cor}

  The kernel function $K(X,Q)$ is defined to be $K(X,Q)=\frac{d\omega^X}{d\omega}(Q)$, the Radon-Nikodym derivative of $\omega^X$ with respect to $\omega$. It satisfies the following estimates:
  \begin{prop}\label{kernelest}
    \begin{enumerate}
      \item If $X$ is a corkscrew point relative to $\Delta_r(Q)$, namely, $X\in\Gamma_{\alpha}(Q)$ with $\abs{X-Q}\approx\dist(X,\bdy\Omega)\approx r$, then $K(X,P)\approx\frac{1}{\omega(\Delta_r(Q))}$ for all $P\in\Delta_r(Q)$.
      \item Let $\Delta_j=\Delta_{2^jr}(Q)$, $R_j=\Delta_j\setminus\Delta_{j-1}$, $j\ge0$. Then
      $$
      \sup_{P\in R_j}K(A_r(Q),P)\le C\frac{2^{-\alpha j}}{\omega(\Delta_j)},
      $$
      where $\alpha=\alpha(n,\Lambda,\lambda,\Gamma,M)$ is the same as in Lemma \ref{BdyHolderCont}.
      \item Let $\Delta=\Delta_r(Q)$ be any boundary disk centered at $Q\in\bdy\Omega$. Then
         $$
          \sup_{P\in\bdy\Omega\setminus\Delta}K(X,P)\rightarrow 0 \qquad\text{as } X\rightarrow Q.
          $$
      \item $K(X,\cdot)$ is a H\"older continuous function.
      Namely, for all $X\in\Omega$,
      $$
      \abs{K(X,Q_1)-K(X,Q_2)}\le C_X\abs{Q_1-Q_2}^{\alpha},
      $$
      where $\alpha>0$ is a constant depending on  $n,\lambda,\Lambda,\Gamma,M,m,m'$, and $C_X>0$ depends additionally on $X$.
    \end{enumerate}
  \end{prop}

  \section{Dirichlet problem with $L^p$ data}\label{LpSec}

  \begin{defn}
    A cone of aperture $\alpha$ is a non-tangential approach region for $Q\in\bdy\Omega$ of the form
    $$
    \Gamma_{\alpha}(Q)=\{X\in\Omega: \abs{X-Q}\le(1+\alpha)\dist(X,\bdy\Omega)\},
    $$
    and a truncated cone is defined by
    $$
    \Gamma_{\alpha}^h(Q)=\Gamma_{\alpha}(Q)\cap B_h(Q).
    $$
    The non-tangential maximal function is defined as
    $$
    u^*(Q)=\sup_{\Gamma_{\alpha}(Q)}\abs{u(X)},
    $$
    and the truncated non-tangential maximal function is defined as
    $$
    u^*_{h}(Q)=\sup_{\Gamma_{\alpha}^h(Q)}\abs{u(X)}.
    $$
  \end{defn}

  By Proposition \ref{kernelest} (1)--(3), we obtain (see \cite{caffarelli1981boundary} or \cite{kenig1994harmonic})
  \begin{lem}\label{ntmaxvsmax}
    Let $\nu$ be a finite Borel measure on $\bdy\Omega$, and
    $u(X)=\int_{\bdy\Omega}K(X,Q)d\nu(Q)$.
    Then for each $P\in\bdy\Omega$,
    $$
    u^*(P)\le C_{\alpha} M_{\omega}(\nu)(P),
    $$
    where $M_{\omega}(\nu)=\sup_{\Delta\ni Q}\frac1{\omega(\Delta)}\abs{\nu}(\Delta)$, and $C_{\alpha}=C(n,\lambda,\Lambda,\Gamma,\alpha)$.
    Moreover, if $\nu\ge0$, then $M_{\omega}(\nu)(P)\le Cu^*(P)$.
  \end{lem}

  \begin{thm}
    Given $f\in L^p(\bdy\Omega,d\omega)$, $1<p\le\infty$, there exists a unique $u$ with $Lu=0$, $u^*\in L^p(\bdy\Omega,d\omega)$, which converges non-tangentially a.e. ($d\omega$) to $f$.
  \end{thm}

  \begin{proof}
    \textbf{Existence.} The existence of the solution can be proved as in the case where $A$ is elliptic, bounded and measurable. Namely, choose $\{f_n\}_{n=1}^{\infty}\subset C(\bdy\Omega)$ such that $f_n\rightarrow f$ in $L^p(\bdy\Omega,d\omega)$.
    Let $u_n$ be the corresponding solution to the Dirichlet problem with boundary data $f_n$. Then $$
    u_n(X)=\int_{\bdy\Omega}f_n(Q)K(X,Q)d\omega,
    $$
    and by Lemma \ref{ntmaxvsmax},
    $$
    u_n^*(P)\lesssim M_{\omega}(f_n)(P) \qquad\forall\, P\in\bdy\Omega.
    $$

    So we have
    \begin{equation}\label{exist_ntmax}
      \norm{(u_n-u_m)^*}_{L^p(d\omega)}\lesssim\norm{M_{\omega}(f_n-f_m)}_{L^p(d\omega)}
      \lesssim\norm{f_n-f_m}_{L^p(d\omega)}\rightarrow0,
      \end{equation}
     as $n,m\rightarrow\infty$.

    Given any compact subset $K$ in $\Omega$, $u_n$ is uniformly Cauchy in $L^{\infty}(K)$. In fact, by the maximal principle, we can assume that the supremum $\sup_{X\in K}\abs{u_n(X)-u_m(X)}$ is attained at some point $Y\in\bdy K$. Then there exists a set $\Delta_Y\subset\bdy\Omega$ such that $\forall\, P\in\Delta_Y$, $Y\in \Gamma_{\alpha}(P)$, and $\sigma(\Delta_Y)\approx\delta(Y)^{n-1}$. By \eqref{exist_ntmax} and the doubling property of $\omega$, we have
    $$
    \sup_{X\in K}\abs{u_n(X)-u_m(X)}\le\frac1{\omega(\Delta_Y)}\int_{\Delta_Y}(u_n-u_m)^*d\omega
    \lesssim\norm{(u_n-u_m)^*}_{L^p(d\omega)}\rightarrow0
    $$
    as $n,m\rightarrow\infty$. Hence $u(X)=\lim_{n\rightarrow\infty}u_n(X)$, for $X\in\Omega$ is pointwise well-defined. From here it is not hard to see that u is the weak solution to $Lu=0$, as well as that $u$ converges non-tangentially a.e. ($d\omega$) to $f$.

    \text{\textbf{Uniqueness.}} Assume $u^*\in L^p(\bdy\Omega,d\omega)$ with $p>1$, and $u$ converges to 0 non-tangentially a.e. ($d\omega$). Then $\int_{\bdy\Omega}u^*_hd\omega\rightarrow 0$ as $h\rightarrow0$. Since $K(X,\cdot)$ is H\"older continuous,
    $$
    \int_{\bdy\Omega}u^*_hd\omega^X\rightarrow 0 \text{ as }h\rightarrow0 \quad\text{ for any fixed} \ X\in\Omega.
    $$
    It suffices to show that for any fixed $X_0\in\Omega$, $u(X_0)=0$.\par
    We first need some geometric observations.
    
    Denote $\Omega_{\delta}=\{Y\in\Omega: \dist(Y,\bdy\Omega)>\delta\}$.
    For $\delta>0$ small, consider the annulus $\overline{\Omega}_{30\delta}\setminus\Omega_{(30+\frac1{\sqrt{n}})\delta}$.
    Cover it by balls $B_{\delta}(Y)$ with $Y\in\bdy\Omega_{30\delta}$. Since the annulus is compact, there exists $N_0=N_0(\delta,\bdy\Omega,n)$ such that
    $$
    \overline{\Omega}_{30\delta}\setminus\Omega_{(30+\frac1{\sqrt{n}})\delta}\subset\bigcup_{k=1}^{N_0}B_{\delta}(Y_k),
    $$
    with $Y_k\in\bdy\Omega_{30\delta}$, $k=1,2,\dots,N_0$. 
    We can find a $C_1=C_1(n,\bdy\Omega)>1$, such that for any $1\le k\le N_0$, there exists a surface ball $\Delta_k\doteq\Delta_{s_k}(Q_k)$ with $\frac{\delta}{C_1}\le s_k\le\delta$ and
    $$
        \abs{Q_k-Y_k}=\dist(Y_k,\bdy\Omega)(=30\delta),
    $$    
    such that
    $$
    B_{20\delta}(Y_k)\subset\Gamma_{\alpha}(P), \quad\forall\, P\in\Delta_k.
    $$
    It is easy to see that for any $X\in B_{20\delta}(Y_k)$ and any $P\in\Delta_k$, $\abs{X-P}\le 60\delta$. So 
    \begin{equation}\label{uniq5}
        B_{20\delta}(Y_k)\subset\Gamma_{\alpha}^{60\delta}(P), \quad\forall\, P\in\Delta_k.
    \end{equation}
    
    Choose a subcollection of disjoint balls $\{{B_{\delta}(Y_{k_j})}\}_{j=1}^{N_1}$ such that
    \begin{equation}
      \bigcup_{k=1}^{N_0}B_{\delta}(Y_k)\subset \bigcup_{j=1}^{N_1}B_{5\delta}(Y_{k_j}),  
    \end{equation}
    for some $N_1=N_1(\delta,n,\bdy\Omega)$. 
    
    We claim that there exists a $\hat{k}<\infty$ depending only on the dimension and the boundary of $\Omega$, such that there are at most $\hat{k}$ overlaps of $\{\Delta_{k_j}\}_{j=1}^{N_1}$. Actually, it is clear that $\abs{Y_{k_j}-Y_{k_l}}\ge 90\delta$ implies $\abs{Q_{k_j}-Q_{k_l}}\ge 30\delta$. So by the disjointness of $\{{B_{\delta}(Y_{k_j})}\}_{j=1}^{N_1}$, the claim is established.
    
    Now we choose $\delta$ sufficiently small so that $X_0\in\Omega_{60\delta}$.
    Let $\vp\in C_0^{\infty}(\Omega_{30\delta})$ with $\vp\equiv1$ in $\Omega_{(30+\frac1{\sqrt{n}})\delta}$, $\abs{\nabla\vp}\le\frac{C}{\delta}$ and $\abs{D^2\vp}\le\frac{C}{\delta^2}$.
     We calculate
    \begin{align}
      u(X_0)&=u(X_0)\vp(X_0)=\int_{\Omega}G(X_0,X)\Big(-A\nabla u\cdot\nabla\vp-A\nabla\vp\cdot\nabla u+(L\vp) u\Big)dX\nonumber\\
      &=-\int A\nabla u\cdot\nabla\vp G+\int A\nabla\vp\cdot\nabla G u\nonumber\\
      &=-\int a\nabla u\cdot\nabla\vp G+\int a\nabla\vp\cdot\nabla G u+
      \Big(\int b\nabla\vp\cdot\nabla uG+\int b\nabla\vp\cdot\nabla Gu\Big)\nonumber\\
      &=-\int a\nabla u\cdot\nabla\vp G+\int a\nabla\vp\cdot\nabla G u+\int b\nabla\vp\cdot\nabla(uG)\nonumber\\
      &\doteq I_1+I_2+I_3,
    \end{align}
    where the second equality follows from Theorem 6.1 of \cite{littman1963regular}, and that $L(u\vp)=-A\nabla u\cdot\nabla\vp-A\nabla\vp\cdot\nabla u+(L\vp) u$. We show that the above expression can be bounded by $\int_{\bdy\Omega}u^*_{60\delta}d\omega^{X_0}$.\par
    
    Let $\{\eta_j\}_{j=1}^{N_1}$ be a partition of unity subordinate to $\{B_{6\delta}(Y_{k_j})\}_{j=1}^{N_1}$, with $\abs{\nabla\eta_j}\le\frac{C}{\delta}$, 
    \begin{equation}\label{POU}
     \sum_{j=1}^{N_1}\eta_j(X)=1 \quad\forall\, X\in \overline{\Omega}_{30\delta}\setminus\Omega_{(30+\frac1{\sqrt{n}})\delta}   
    \end{equation}

    Using Caccioppli's inequality, Harnack principle, \eqref{uniq5}, Corollary \ref{omegaG}, and the bounded overlaps of  $\{\Delta_{k_j}\}_{j=1}^{N_1}$, we estimate
    \begin{align*}
      \abs{I_1}&=\sum_{j=1}^{N_1}\abs{\int_{\Omega}\eta_j(X)a(X)\nabla u(X)\cdot\nabla\vp(X) G(X_0,X)dX}\\
      &\lesssim\delta^{-1}\sum_{j=1}^{N_1}\int_{B_{6\delta}(Y_{k_j})}\abs{\nabla u(X)}G(X_0,X)dX\\
      &\lesssim\delta^{\frac{n}{2}-1}\sum_{j=1}^{N_1}\Big(\int_{B_{6\delta}(Y_{k_j})}\abs{\nabla u}^2\Big)^{1/2}\sup_{X\in B_{6\delta}(Y_{k_j})}G(X_0,X)\\
     &\lesssim \delta^{\frac{n}{2}-2}\sum_{j=1}^{N_1}\Big(\int_{B_{10\delta}(Y_{k_j})}\abs{u}^2\Big)^{1/2}G(X_0,A_{\delta}(Q_{k_j}))\\    &\lesssim\sum_{j=1}^{N_1}\inf_{P\in\Delta_{k_j}}u^*_{60\delta}(P)
      \omega^{X_0}(\Delta_{\delta}(Q_{k_j}))\\
      &\lesssim \sum_{j=1}^{N_1}\int_{\Delta_{k_j}}u^*_{60\delta}(P)d\omega^{X_0}
      \lesssim\int_{\bdy\Omega}u^*_{60\delta}d\omega^{X_0},
    \end{align*}
    where $A_{\delta}(Q_{k_j})\in\Gamma_{\alpha}(Q_{k_j})$ is the corkscrew point with $\dist(A_{\delta}(Q_{k_j}),\bdy\Omega)\approx\delta$.

    $I_2$ can be estimated in a similar manner. To estimate $I_3$, 
    we write it as
    $$
      I_3=\int_{\Omega}b\nabla(\sum_{j=1}^{N_1}\eta_j\vp)\cdot\nabla(uG)
      =\sum_{j=1}^{N_1}\int_{\Omega}b\nabla(\eta_j\vp)\cdot\nabla(\xi_j uG),
    $$
    where $\xi_j\in C_0^{\infty}(B_{7\delta}(Y_{k_j}))$ with $\xi_j\equiv1$ in $B_{6\delta}(Y_{k_j})$ and $\abs{\nabla\xi_j}\le C/{\delta}$. We have used here the support property of $\nabla\vp$ and \eqref{POU}.\par
    
    By Sobolev inequality, $uG\xi_j\in W^{1,p}(\Rn)$ for some $1<p\le\frac{n}{n-1}$. Then we can apply Proposition \ref{compensatedProp} and get
    \begin{align}\label{I3}
      &\abs{I_3}\lesssim\Gamma
      \sum_{j=1}^{N_1}\norm{\nabla(\eta_j\vp)}_{L^{p'}}\norm{\nabla(uG\xi_j)}_{L^p}\nonumber\\
      &\lesssim\Gamma\sum_{j=1}^{N_1}\delta^{\frac{n}{p'}-1}\Big\{\Big(\int_{B_{7\delta}(Y_{k_j})}\abs{u\nabla G}^p\Big)^{1/p}+\Big(\int_{B_{7\delta}(Y_{k_j})}\abs{G\nabla u}^p\Big)^{1/p}+\delta^{-1}\Big(\int_{B_{7\delta}(Y_{k_j})}\abs{uG}^p\Big)^{1/p}\Big\}
    \end{align}
      Now we can estimate $I_3$ as we did for $I_1$, namely,
    \begin{align*}
      \delta^{\frac{n}{p'}-1}\Big(\int_{B_{7\delta}(Y_{k_j})}\abs{u\nabla G}^p\Big)^{1/p}
      &\lesssim\delta^{\frac{n}{p'}-1}\Big(\int_{B_{7\delta}(Y_{k_j})}\abs{\nabla G}^2\Big)^{1/2}\Big(\int_{B_{7\delta}(Y_{k_j})}\abs{u}^{\frac{2p}{2-p}}\Big)^{\frac{2-p}{2p}}\\
      &\lesssim\delta^{n-2}\sup_{X\in B_{7\delta}(Y_{k_j})}G(X_0,X)\inf_{P\in \Delta_{k_j}}u^*_{60\delta}(P)\\
      &\lesssim\int_{\Delta_{k_j}}u^*_{60\delta}d\omega^{X_0},
    \end{align*}
      while the other two integrals in the right-hand side of \eqref{I3} can be estimated similarly. Thus
      $$
      \abs{I_3}\lesssim\Gamma\sum_{j=1}^{N_1}\int_{\Delta_{k_j}}u^*_{60\delta}d\omega^{X_0}\lesssim\Gamma\int_{\bdy\Omega}u^*_{60\delta}d\omega^{X_0}.
      $$
where in the last inequality we have used that there are at most $\hat{k}$ overlaps of $\{\Delta_{k_j}\}_{j=1}^{N_1}$.
    Altogether we obtain $\abs{u(X_0)}\le C(n,\bdy\Omega)(1+\Gamma)\int_{\bdy\Omega}u^*_{60\delta}d\omega^{X_0}$. Letting $\delta$ go to zero, we get $u(X_0)=0$.
  \end{proof}
  
  \begin{re}
    We can show that if $f\in L^1(\bdy\Omega,d\omega)$, then there exists a $u$ with $Lu=0$, $u^*\in L^{1,\infty}(\Omega,d\omega)$ which converges non-tangentially a.e. ($d\omega$) to $f$. But such $u$ is not unique. 
  \end{re}

  \section{Square function estimates}\label{squareSec}
  \begin{defn}
    The square function with aperture determined by $\alpha$ is defined as
    $$
    S_{\alpha}u(Q)=\Big(\iint_{\Gamma_{\alpha}(Q)}\abs{\nabla u(X)}^2\delta(X)^{2-n}dX\Big)^{1/2}.
    $$
    The truncated square function $S_{\alpha,h}u(Q)$ is defined similarly, integrating over the truncated cone $\Gamma^h_{\alpha}(Q)$.
  \end{defn}

  The main lemma of this section is the following:
  \begin{lem}\label{LemIden}
    Let $f\in L^2(\bdy\Omega,d\omega)$, $u(X)=\int_{\bdy\Omega}f(Q)d\omega^X(Q)$, where $\omega$ is the elliptic measure associated to the operator $L$ satisfying \eqref{aBound}--\eqref{bBMO}. Let $X_0\in\Omega$. Then
    \begin{equation}\label{identity}
      \int_{\Omega}A(Y)\nabla u(Y)\cdot\nabla u(Y) G(X_0,Y)dY=-\frac1{2}u(X_0)^2+\frac1{2}\int_{\bdy\Omega}f(Q)^2d\omega^{X_0}(Q).
    \end{equation}

  \end{lem}
This lemma has been proved for harmonic functions $u$ in an NTA domain using a classical potential theoretic result of Riesz. See Theorem 5.14 in \cite{jerison1982boundary} for details. We present here a proof that does not use the Riesz theorem, and for solutions $u$ to $Lu=0$ in $\Omega$.\par
Once the lemma is proved, an argument in \cite{dahlberg1984area} can be verified to work for operators $L=-\divg A\nabla$ where the anti-symmetric part of $A$ belongs to $BMO$. Thus we are able to obtain

\begin{thm}
  Let $\Omega$ be a bounded NTA domain in $\Rn$, $n\ge3$. Let $\mu$ be a positive measure satisfying $A_\infty$ with respect to $\omega^{X_0}$. Let $u$ be the solution to $Lu=0$ in $\Omega$. Then for all $0<p<\infty$,
  $$
  \Big(\int_{\bdy\Omega}(S_{\alpha}(u))^pd\mu\Big)^{1/p}\le C\Big(\int_{\bdy\Omega}(u^*)^pd\mu\Big)^{1/p}
  $$
  where $C=C(n,\lambda,\Lambda,\Gamma,p,\alpha)$. If, in addition, $u(X_0)=0$, then
  $$
  \Big(\int_{\bdy\Omega}(u^*)^pd\mu\Big)^{1/p}\le C\Big(\int_{\bdy\Omega}(S_{\alpha}(u))^pd\mu\Big)^{1/p}
  $$
  where $C=C(n,\lambda,\Lambda,\Gamma,p,\alpha)$.

\end{thm}

Before proving Lemma\ \ref{LemIden}, we provide Green's formula in the following form:

\begin{lem}[Green's formula]\label{LemGreenform}
Let $\Omega$ be a domain with $\bdy\Omega\in C^2$. Let $A$ be a matrix with entries belong to $C^{0,1}(\Rn)$. Let $L=-\divg A\nabla$, and the adjoint operator $L^*=-\divg A^*\nabla$. Let $G^*$ be the Green function corresponding to $L^*$. Then for all $u\in W^{1,p'}(\Omega)\cap C(\Omega)$ with $1<p<\frac{n}{n-1}$ and $\frac1{p'}+\frac1{p}=1$,
\begin{equation}
u(X)=\int_{\Omega}A^*(Y)\nabla G^*(Y,X)\cdot\nabla u(Y)dY-\int_{\bdy\Omega}A^*(Q)\nabla G^*(Q,X)\cdot \vec{N}(Q)u(Q)d\sigma(Q)
\end{equation}
$\forall\,X\in\Omega$, where $N(Q)$ is the outward unit normal at $Q\in\bdy\Omega$.

\end{lem}

\begin{proof}
  For any $\epsilon>0$, denote $\Omega_{\epsilon}=\{X\in\Omega: \delta(X)>\epsilon\}$. Choose $\eta_{\epsilon}\in C_0^{\infty}(\Omega)$ such that $0\le\eta_{\epsilon}\le1$, and that
  $$
  \eta_{\epsilon}=
  \begin{cases}
    1\qquad\text{in}\quad\Omega_{\epsilon},\\
    0\qquad\text{in}\quad\Omega\setminus\Omega_{\frac{\epsilon}{2}}.
  \end{cases}
  $$
  Let $\epsilon$ be sufficiently small, then
  \begin{align*}
    &\int_{\Omega}A^*(Y)\nabla G^*(Y,X)\cdot\nabla u(Y)dY=
    \int_{\Omega}A^*\nabla G^*\cdot\nabla\Big((1-\eta_{\epsilon})u\Big)dY+\int_{\Omega}A^*\nabla G^*\cdot\nabla(\eta_{\epsilon}u)dY\\
    &=\int_{\Omega\setminus\Omega_{2\epsilon}}A^*\nabla G^*\cdot\nabla\Big((1-\eta_{\epsilon})u\Big)dY+u(X)\\
    &=-\int_{\Omega\setminus\Omega_{2\epsilon}}\divg(A^*\nabla G^*)(1-\eta_{\epsilon})udY
    +\int_{\bdy\Omega}A^*\nabla G^*\cdot \vec{N}ud\sigma+u(X)
  \end{align*}
  The assumptions on the regularity of the boundary and the coefficients imply that
  $G^*\in W^{2,2}(\Omega\setminus\Omega_{2\epsilon})$, so that $\divg(A^*\nabla G^*)$ is uniformly bounded in $\epsilon$. Thus, as $\epsilon\rightarrow0$,
  $$
  -\int_{\Omega\setminus\Omega_{2\epsilon}}\divg(A^*\nabla G^*)(1-\eta_{\epsilon})udY\rightarrow 0.
  $$
   Then the lemma follows.

\end{proof}
\begin{re}
  From the proof one can see that this lemma is also true for $u\in W^{1,2}(\Omega)\cap W_{\loc}^{1,p'}(\Omega)\cap C(\Omega)$. Also, note that if $u\in W_0^{1,p'}(\Omega)\cap C(\Omega)$, then the lemma is simply the definition of Green function.
\end{re}

\begin{cor}\label{CorMeasureRep}
  Let $A$ be a $n\times n$ matrix with entries belong to $C^{k,1}(\Rn)$ with $k\ge0$ and $k>\frac{n}{2}-2$. Let $\Omega$, $L\ ,L^*\ ,G^*$ be as in Lemma\ \ref{LemGreenform}. Then
  \begin{equation}
    d\omega_L^X(Q)=-A^*(Q)\nabla G^*(Q,X)\cdot \vec{N}(Q)d\sigma.
  \end{equation}
\end{cor}
\begin{proof}
  For any $\vp\in Lip(\bdy\Omega)$, let $u\in W^{1,2}(\Omega)$ be the weak solution of
  $$
  \begin{cases}
    Lu=0 \qquad\text{in}\quad\Omega,\\
    u=\vp \qquad\text{on}\quad\bdy\Omega.
  \end{cases}
  $$
  Since $a_{ij}\in C^{k,1}(\Rn)$, $u\in W_{loc}^{k+2,2}(\Omega)\cap C(\overline{\Omega})$. By the Sobolev inequality, $u\in W_{loc}^{1,p'}(\Omega)$ for some $1<p<\frac{n}{n-1}$ and $\frac1{p'}+\frac1{p}=1$. Then
  $$
  \int_{\Omega}A^*(Y)\nabla G^*(Y,X)\cdot\nabla u(Y)dY=\int_{\Omega}A(Y)\nabla u(Y)\cdot\nabla G(X,Y)dY=0.
  $$
  By Lemma\ \ref{LemGreenform},
  $$
  u(X)=-\int_{\bdy\Omega}A^*(Q)\nabla G^*(Q,X)\cdot \vec{N}(Q)\vp(Q) d\sigma(Q),
  $$
  $\Rightarrow$ $\int_{\bdy\Omega}\vp(Q)d\omega_L^X(Q)=-\int_{\bdy\Omega}A^*(Q)\nabla G^*(Q,X)\cdot \vec{N}(Q)\vp(Q) d\sigma(Q)$.

  Then a standard limiting argument completes the proof.
\end{proof}

We now prove Lemma \ref{LemIden}.
\begin{proof}
  We first point out that the expression $A\nabla u\cdot\nabla uG$ equals $a\nabla u\cdot\nabla u G$ a.e. since $b\nabla u\cdot\nabla u=0$ a.e.. In other words, the integral on the left-hand side of \eqref{identity} is actually $\int_{\Omega}a(Y)\nabla u(Y)\cdot\nabla u(Y) G(X_0,Y)dY$.

  Let $X_0\in\Omega$ be any fixed point. We prove the lemma in 4 steps.

  \underline{Step 1.} Assume $\bdy\Omega\in C^{2}$, $A=(a_{ij})$ with $a_{ij}\in C^{\infty}(\Rn)$, $(A\xi)\cdot\xi\ge\lambda\abs{\xi}^2$. Let $f\in C(\bdy\Omega)$. Then \eqref{identity} holds.\par
  To see this, we calculate, by Corollary \ref{CorMeasureRep}
  \begin{align*}
  -\int_{\Omega}\divg\Big(u^2(Y)A^*(Y)\nabla G^*(Y,X_0)\Big)dY&=-\int_{\bdy\Omega}f^2A^*(Q)\nabla G^*(Q,X_0)\cdot \vec{N}(Q)d\sigma(Q)\\
  &=\int_{\bdy\Omega}f^2d\omega^{X_0}.
  \end{align*}
  On the other hand,
  \begin{align*}
    &-\int_{\Omega}\divg\Big(u^2(Y)A^*(Y)\nabla G^*(Y,X_0)\Big)dY=
    -\int_{\Omega}\divg(A^*\nabla G^*)u^2dY-2\int_{\Omega}u\nabla u\cdot(A^*\nabla G^*)dY\\
    &=u^2(X_0)-2\int_{\Omega}u(Y)\Big(A(Y)\nabla u(Y)\Big)\cdot\nabla G(X_0,Y)dY\\
    &=u^2(X_0)+2\int_{\Omega}A(Y)\nabla u(Y)\cdot\nabla u(Y)G(X_0,Y)dY.
  \end{align*}
  Then \eqref{identity} follows.

 \underline{Step 2.} Let $\Omega$ a bounded NTA domain. Then there exists a sequence of domains $\{\Omega_t\}$ with $\bdy\Omega_t\in C^{2}$, $\Omega_t\subset\subset\Omega$ such that $\Omega_t$ increases to $\Omega$ as $t\rightarrow\infty$. Let $A=(a_{ij})$ with $a_{ij}\in C^{\infty}(\Rn)$, $(A\xi)\cdot\xi\ge\lambda\abs{\xi}^2$, and let $f\in Lip(\bdy\Omega)$.

 Let $u$ be the solution to $Lu=0$ in $\Omega$ with boundary data $f$. Then $u\in C(\overline{\Omega})\cap C^{\infty}(\Omega)$, and in particular, $u\in C^{\infty}(\bdy\Omega_t)$. So by the result we obtained in Step 1, we have
 \begin{equation}
   \int_{\Omega_t}A(Y)\nabla u(Y)\cdot\nabla u(Y) G_t(X_0,Y)=-\frac1{2}u^2(X_0)+\frac1{2}\int_{\bdy\Omega_t}u^2d\omega_t^{X_0},
 \end{equation}
 where $G_t$ is the Green function for domain $\Omega_t$, and $\omega_t$ is the elliptic measure on domain $\Omega_t$.

 Extending $G^*_t$ to be zero in $\Omega\setminus\Omega_t$. We still denote this extension as $G^*_t$.

 We now discuss the convergence of $G_t^*$. By the estimate for Green function, we have
\begin{equation}\label{GtBound}
\norm{G^*_t}_{W^{1,k}(\Omega)}\le C_k \qquad\text{for all }t,
\end{equation}
where $1<k<\frac{n}{n-1}$. So there exists a sequence $t_n\rightarrow\infty$ such that $G_{t_n}^*$ converges to a $\tilde{G}^*$ weakly in $W_0^{1,k}$,
$\Rightarrow$ there exists a subsequence, still denoted by $t_n$ such that $G_{t_n}^*\rightarrow \tilde{G}^*$ in $L^p(\Omega)$ for all $1<p<\frac{nk}{n-k}$. And then there is a subsequence converges to $\tilde{G}^*$ a.e. in $\Omega$.
For simplicity, we write
 \begin{equation}\label{GtConv}
    G^*_t\rightarrow \tilde{G}^* \ \text{weakly in }W_0^{1,k}(\Omega),\ \text{strongly in }L^p(\Omega)\quad\forall\, 1<p<\frac{nk}{n-k},\ \text{and a.e. in }\Omega
 \end{equation}
 for some $\tilde{G}^*$.

\underline{Claim 1.} The limit $\tilde{G}^*$ is the Green function corresponding to $L^*$ on $\Omega$.

 By a limiting process, it suffices to show $\forall\,\vp\in C_0^1(\Omega)$,
 \begin{equation}
   \int_{\Omega}A^*(Y)\nabla \tilde{G}^*(Y,X)\cdot\nabla\vp(Y)dY=\vp(X), \qquad\forall\, X\in\Omega.
 \end{equation}
For any fixed $X\in\Omega$, choose $t$ to be sufficiently large so that $X\in\Omega_t$ and that $\supp\vp\subset\Omega_t$. Then
$$
\int_{\Omega}A^*(Y)\nabla G^*_t(Y,X)\cdot\nabla\vp(Y)dY=\int_{\Omega_t}A^*(Y)\nabla G^*_t(Y,X)\cdot\nabla\vp(Y)dY=\vp(X).
$$
By the weak convergence of $G_t^*$ \eqref{GtConv}, we obtain
$$
\int_{\Omega}A^*(Y)\nabla \tilde{G}^*(Y,X)\cdot\nabla\vp(Y)dY=\lim_{t\rightarrow\infty}\int_{\Omega}A^*(Y)\nabla G^*_t(Y,X)\cdot\nabla\vp(Y)dY=\vp(X).
$$
By uniqueness of Green function, $\tilde{G}^*=G^*$ is the Green function on $\Omega$.

Let $v_t(X)=\int_{\bdy\Omega_t}u^2d\omega_t^X$. Then $Lv_t=0$ in $\Omega_t$, $v_t|_{\bdy\Omega_t}=u^2$. Let $v(X)=\int_{\bdy\Omega}f^2d\omega^{X}$.

We have $u,v\in C(\overline{\Omega})$ and $v_t\in C(\overline{\Omega_t})$. So for any $\epsilon>0$, $\exists\, t_1$ such that whenever $t\ge t_1$,
$
\sup_{\bdy\Omega_t}\abs{v_t-v}<\epsilon
$.
Then by maximum principle,
$$
\sup_{\Omega_t}\abs{v_t-v}<\epsilon\qquad\forall\, t\ge t_1.
$$
Therefore,
$$
\lim_{t\rightarrow\infty}\int_{\bdy\Omega_t}u^2d\omega_t^{X_0}=\int_{\bdy\Omega}f^2d\omega^{X_0}.
$$
Recall that we have
$$
\int_{\Omega}A(Y)\nabla u(Y)\cdot\nabla u(Y) G_t(X_0,Y)=-\frac1{2}u^2(X_0)+\frac1{2}\int_{\bdy\Omega_t}u^2d\omega_t^{X_0},
$$
so
\begin{equation}
  \lim_{t\rightarrow\infty}\int_{\Omega}A\nabla u\cdot\nabla uG_t=-\frac1{2}u^2(X_0)+\frac1{2}\int_{\bdy\Omega}f^2d\omega^{X_0}<\infty.
\end{equation}

By this, \eqref{GtConv} and Claim 1, Fatou's lemma implies that
$
\int_{\Omega}A\nabla u\cdot\nabla uG<\infty.
$
Therefore, for any $\epsilon>0$, $\exists\, r_0$ such that whenever $r\le r_0$,
\begin{equation}\label{I4}
  \abs{\int_{B_r(X_0)}A\nabla u\cdot\nabla uG}<\epsilon.
\end{equation}

\underline{Claim 2.}
\begin{equation}
  \lim_{t\rightarrow\infty}\int_{\Omega}A\nabla u\cdot\nabla u G_t=\int_{\Omega}A\nabla u\cdot\nabla uG.
\end{equation}
We write
\begin{align*}
&\abs{\int_{\Omega}A\nabla u\cdot\nabla u G_t-\int_{\Omega}A\nabla u\cdot\nabla uG}\\
&\le\abs{\int_{B_r(X_0)}A\nabla u\cdot\nabla u G}
+\abs{\int_{B_r(X_0)}A\nabla u\cdot\nabla u G_t}
+\abs{\int_{\Omega\setminus B_r(X_0)}A\nabla u\cdot\nabla u (G_t-G)}\\
&\doteq I_1+I_2+I_3.
\end{align*}

For any $\epsilon>0$, choose $r$ to be as in \eqref{I4}, so $I_1<\epsilon$.

By \eqref{GW1.8} and \eqref{GW1.9}, we have
  $$
    G_t(X_0,Y)\le C(n,\lambda,\Lambda,\Gamma)G(X_0,Y) \qquad\forall\, Y\text{ such that }\abs{X_0-Y}\le\frac1{2}\delta(X_0).
  $$
  So
  $$
  I_2\le CI_1< C\epsilon,
  $$
  where $C=C(n,\lambda,\Lambda,\Gamma)$.

  For $I_3$, since $\abs{G_t-G}$ is bounded in $\Omega\setminus B_r(X_0)$, we can use Fatou's lemma to get
  $$
  \lim_{t\rightarrow\infty}\int_{\Omega\setminus B_r(X_0)}A\nabla u\cdot\nabla u \abs{G_t-G}=0
  $$

Thus the claim is established, and we have shown that \eqref{identity} holds for NTA domain $\Omega$, $L=-\divg A\nabla$ with $A$ smooth, and $f\in Lip(\bdy\Omega)$.

 \underline{Step 3.} Let $A_s=a^s+b^s=(a^s_{ij})+(b^s_{ij})$, where
  $a^s_{ij}(X)=a_{ij}*\alpha_s(X)\in C^{\infty}(\Rn)$ and $b^s_{ij}(X)=b_{ij}*\beta_s(X)\in C^{\infty}(\Rn)$. $a^s$ and $b^s$ satisfy: for all $s$,
  \begin{equation}\label{abNormBd}
    (a^s\xi)\cdot\xi\ge\lambda\abs{\xi}^2, \quad \norm{a_{ij}^s}_{L^{\infty}}\le\Lambda,\quad
    \norm{b^s}_{BMO(\Omega)}\le\norm{b}_{BMO(\Omega)}.
  \end{equation}
  And as $s\rightarrow\infty$,
  \begin{equation}\label{abCovLp}
    a^s\rightarrow a\quad\text{in}\ L^p(\Omega), \quad b^s\rightarrow b\quad\text{in}\ L^p(\Omega), \forall\, p>1,
  \end{equation}
 \begin{equation}\label{abCovae}
   a^s\rightarrow a\quad\text{a.e. in}\ \Omega, \quad b^s\rightarrow b\quad\text{a.e. in}\ \Omega.
 \end{equation}
 We want to show \eqref{identity} for NTA domain $\Omega$, $A=a+b$ satisfying \eqref{aBound}--\eqref{bBMO}, and $f\in Lip(\bdy\Omega)$.

 Let $L_s=-\divg A_s\nabla$. Let $u_s$ be the solution to
 $$
 \begin{cases}
 L_su_s=0\qquad\text{in }\Omega,\\
 u_s=f\qquad\text{on }\bdy\Omega.
 \end{cases}
 $$
 Let $G_s$ and $\omega_s$ be the Green function and the elliptic measure associated to $L_s$, respectively. Then from Step 2., it follows that
 \begin{equation}\label{identitys}
   \int_{\Omega}A_s(Y)\nabla u_s(Y)\cdot\nabla u_s(Y)G_s(X_0,Y)dY=-\frac1{2}u_s^2(X_0)+\frac1{2}\int_{\bdy\Omega}f^2d\omega_s^{X_0}.
 \end{equation}

 Let $v_s(X)=\int_{\bdy\Omega}f^2d\omega_s^X$. Then $L_sv_s=0$ in $\Omega$, $v_s\in W^{1,2}(\Omega)\cap C(\overline{\Omega})$ and $v_s|_{\bdy\Omega}=f^2$. Let $v(X)=\int_{\bdy\Omega}f^2d\omega^X$.
 Using Remark \ref{RHgrad_global}, one can prove as in \cite{kenig1993neumann} 7.1--7.5, that
 \begin{equation}\label{vsW12Conv}
   \norm{v_s-v}_{W^{1,2}(\Omega)}\rightarrow0 \qquad\text{as }s\rightarrow\infty.
 \end{equation}
 
 Also, since $v_s$'s are equicontinuous (by Lemma \ref{intHolderLem}, Lemma \ref{BdyHolderCont} and \eqref{abNormBd}) and uniformly bounded (by the maximum principle),
 \begin{equation}
   v_s\rightarrow v \quad\text{uniformly in }\overline{\Omega}.
 \end{equation}
 For the same reason,
 \begin{equation}
   \norm{u_s-u}_{W^{1,2}(\Omega)}\rightarrow0\quad\text{and}\quad
   u_s\rightarrow u \quad\text{uniformly in }\overline{\Omega}.
 \end{equation}
Thus,
\begin{align}\label{I2est1}
  &\lim_{s\rightarrow\infty}\int_{\Omega}A_s\nabla u_s\cdot\nabla u_sG_s=
  \lim_{s\rightarrow\infty}\Big(-\frac1{2}u_s^2(X_0)+\frac1{2}\int_{\bdy\Omega}f^2d\omega_s^{X_0}\Big)\nonumber\\
  &=-\frac1{2}u^2(X_0)+\frac1{2}\int_{\bdy\Omega}f^2d\omega^{X_0}<\infty.
\end{align}
Now consider the convergence of $G_s^*$. By \eqref{GW1.7} and \eqref{abNormBd},
\begin{equation}\label{GsBound}
\norm{G^*_s}_{W^{1,k}(\Omega)}\le C_k \qquad\text{for all }s,
\end{equation}
where $C_k=C(n,\lambda,\Lambda,\Gamma,k)$ and $1<k<\frac{n}{n-1}$.
Then the same argument about the convergence of $G_t^*$ in Step 2 applies to $G^*_s$. So we have
\begin{equation}\label{GsConv}
  G^*_s\rightarrow G^* \ \text{weakly in }W_0^{1,k}(\Omega),\ \text{strongly in }L^p(\Omega)\quad\forall\, 1<p<\frac{nk}{n-k},\ \text{and a.e. in }\Omega
\end{equation}
for some $G^*$.

\underline{Claim 3.} The limit $G^*$ is the Green function corresponding to $L^*$ in $\Omega$.

By a limiting process, it suffices to show $\forall\,\vp\in C_0^1(\Omega)$,
\begin{equation}\label{G*Green}
\int_{\Omega}A^*(X)\nabla G^*(X,Y)\cdot\nabla\vp(X)dX=\vp(Y).
\end{equation}

We have
\begin{equation}
\abs{\int_{\Omega}A^*_s\nabla G^*_s\cdot\nabla\vp-\int_{\Omega}A^*\nabla G^*\cdot\nabla\vp}
\le\abs{\int_{\Omega}A^*_s\nabla(G^*_s-\nabla G^*)\cdot\nabla\vp}+\abs{\int_{\Omega}(A_s^*-A^*)\nabla G^*\cdot\nabla\vp}
\end{equation}
The first integral on the right-hand side is bounded by $(\Lambda+\Gamma)\norm{G^*_s-G^*}_{W^{1,k}(\Omega)}\norm{\vp}_{W^{1,k'}(\Omega)}$, while the second integral is bounded by
$$
\norm{A_s-A}_{L^{k_0'}(\Omega)}\Big(\int_{\Omega}\abs{\nabla G^*\cdot\nabla\vp}^{k_0}\Big)^{1/{k_0}}
$$
for some $1<k_0<k$. Thus by \eqref{abCovLp} and \eqref{GsConv},
\begin{equation}
  \lim_{s\rightarrow\infty}\int_{\Omega}A^*_s\nabla G^*_s\cdot\nabla\vp=\int_{\Omega}A^*\nabla G^*\cdot\nabla\vp.
\end{equation}
Since
$$
\int_{\Omega}A_s^*(X)\nabla G_s^*(X,Y)\cdot\nabla\vp(X)dX=\vp(Y),
$$
we have proved \eqref{G*Green}. Then by uniqueness of Green function, $G^*$ is the Green function corresponding to $L^*$ in $\Omega$.

\underline{Claim 4.} \begin{equation}
  \lim_{s\rightarrow\infty}\int_{\Omega}A_s(Y)u_s(Y)\cdot\nabla u_s(Y)G_s(X_0,Y)dY=\int_{\Omega}A(Y)u(Y)\cdot\nabla u(Y)G(X_0,Y)dY.
\end{equation}
  Write
  \begin{align*}
    &\abs{\int_{\Omega}A_s(Y)u_s(Y)\cdot\nabla u_s(Y)G_s(X_0,Y)dY-\int_{\Omega}A(Y)u(Y)\cdot\nabla u(Y)G(X_0,Y)dY}\\
    &\le\abs{\int_{B_{\frac{\delta}{2}}(X_0)}A_s\nabla u_s\cdot\nabla u_sG_s}
    +\abs{\int_{B_{\frac{\delta}{2}}(X_0)}A\nabla u\cdot\nabla uG}\\
    &\quad+\abs{\int_{\Omega\setminus B_{\frac{\delta}{2}}(X_0)}A_s\nabla u_s\cdot\nabla u_sG_s
    -\int_{\Omega\setminus B_{\frac{\delta}{2}}(X_0)}A\nabla u\cdot\nabla uG}\\
    &\doteq I_4+I_5+I_6.
  \end{align*}
  For $I_4$, applying the identity \eqref{identitys} to the ball $B_{\delta}(X_0)$ to obtain
  $$
  \int_{B_{\delta}(X_0)}A_s\nabla u_s\cdot\nabla u_sG_{s,B_{\delta}(X_0)}dY
  =-\frac1{2}u_s^2(X_0)
  +\frac1{2}\int_{\abs{X_0-Y}=\delta}u_s^2d\omega_{s,B_{\delta}(X_0)}^{X_0},
  $$
  where $G_{s,B_{\delta}(X_0)}$ is the Green function, and $\omega_{s,B_{\delta}(X_0)}$ is the elliptic measure, associated to $L_s$ and domain $B_{\delta}(X_0)$.

  By \eqref{GW1.8}, \eqref{GW1.9} and \eqref{abNormBd}, we have
  \begin{equation}
    G_s(X_0,Y)\le C(n,\lambda,\Lambda,\Gamma)G_{s,B_{\delta}(X_0)}(X_0,Y) \qquad\forall\, Y\in B_{\frac{\delta}{2}}(X_0).
  \end{equation}
  Also note that $A_s\nabla u_s\cdot\nabla u_sG_s=a_s\nabla u_s\cdot\nabla u_sG_s\ge0$, so
  \begin{align}\label{I1est}
    &\abs{\int_{B_{\frac{\delta}{2}}(X_0)}A_s\nabla u_s\cdot\nabla u_sG_s}
    \le C\abs{\int_{B_{\delta}(X_0)}A_s\nabla u_s\cdot\nabla u_sG_{s,B_{\delta}(X_0)}dY}\nonumber\\
    &= C\abs{-\frac1{2}u_s^2(X_0)
  +\frac1{2}\int_{\abs{X_0-Y}=\delta}u_s(Y)^2d\omega_{s,B_{\delta}(X_0)}^{X_0}(Y)}\nonumber\\
  &=C\abs{\int_{\abs{X_0-Y}=\delta}\Big(u_s^2(Y)-u_s^2(X_0)\Big)d\omega_{s,B_{\delta}(X_0)}^{X_0}}\nonumber\\
  &\le C\sup_{\abs{X_0-Y}=\delta}\abs{u_s^2(Y)-u_s^2(X_0)}
  \le C\sup_{\bdy\Omega}\abs{f}\sup_{\abs{X_0-Y}=\delta}\abs{u_s(Y)-u_s(X_0)}\nonumber\\
  &\le C\sup_{\bdy\Omega}\abs{f}^2\delta^{\alpha},
  \end{align}
  where $C=C(n,\lambda,\Lambda,\Gamma)$, $0<\alpha=\alpha(n,\lambda,\Lambda,\Gamma)<1$, and we used Lemma \ref{intHolderLem} to obtain the last inequality.

  For $I_5$, it follows from Fatou's lemma, the convergence of $G_s^*$ and $u_s$, as well as \eqref{I2est1} that
  \begin{align}
    0&\le\int_{\Omega}A\nabla u\cdot\nabla uG
    =\int_{\Omega}\liminf\limits_{s\rightarrow\infty}A_s\nabla u_s\cdot\nabla u_s G_s
    \le \liminf\limits_{s\rightarrow\infty}\int_{\Omega}A_s\nabla u_s\cdot\nabla u_s G_s\nonumber\\
    &=-\frac1{2}u^2(X_0)+\frac1{2}\int_{\bdy\Omega}f^2d\omega^{X_0}<\infty.
  \end{align}
  $\Rightarrow$ $\forall\, \epsilon>0$, there exists a $\delta_0>0$, such that whenever $\delta<\delta_0$,
  $$
  \abs{\int_{B_{\frac{\delta}{2}}(X_0)}A\nabla u\cdot\nabla uG}<\epsilon, \text{ i.e. } I_5<\epsilon.
  $$
  By \eqref{I1est}, we can choose $\delta<\delta_0$ to be so small that $I_4<\epsilon$. Note that $\delta$ is independent of $s$. Let this $\delta$ be fixed.

  We now estimate $I_6$.
  \begin{align*}
    I_6&\le\abs{\int_{\Omega\setminus B_{\frac{\delta}{2}}(X_0)}A_s\nabla u_s\cdot \nabla u_s (G_s-G)}
    +\abs{\int_{\Omega\setminus B_{\frac{\delta}{2}}(X_0)}(A_s\nabla u_s\cdot \nabla u_s-A\nabla u\nabla u)G}\\
    &\doteq I_s+II_s.
  \end{align*}
  For $I_s$,
  $$
  I_s=\abs{\int_{\Omega\setminus B_{\frac{\delta}{2}}(X_0)}a_s\nabla u_s\cdot \nabla u_s (G_s-G)}\le\Lambda\int_{\Omega\setminus B_{\frac{\delta}{2}}(X_0)}\abs{\nabla u_s}^2\abs{G_s-G}
  $$
  Since $G_s$ and $G$ are bounded in $\Omega\setminus B_{\frac{\delta}{2}}(X_0)$, $\abs{\nabla u_s}^2\abs{G_s-G}$ is uniformly integrable. Then Vitali convergence theorem gives that
  $\lim_{s\rightarrow\infty}I_s=0.$

  For $II_s$, we have
  \begin{align*}
    II_s&\le\sup_{\Omega\setminus B_{\frac{\delta}{2}}(X_0)}G(X_0,Y)
    \int_{\Omega\setminus B_{\frac{\delta}{2}}(X_0)}\abs{a_s\nabla u_s\cdot\nabla u_s-a\nabla u\cdot\nabla u}\\
    &\le\Lambda\sup_{\Omega\setminus B_{\frac{\delta}{2}}(X_0)}G(X_0,Y)\int\abs{\nabla u_s\cdot\nabla u_s-\nabla u\cdot\nabla u}+\int\abs{a_s-a}\abs{\nabla u}^2.
  \end{align*}
  So $\lim_{s\rightarrow\infty}II_s=0$. Note that $\delta$ has been fixed when estimating $I_s$ and $II_s$.

  Therefore, we have shown that for any $\epsilon>0$, there exists $\delta>0$ such that
  $$
  \lim_{s\rightarrow\infty}\abs{\int_{\Omega}A_s(Y)u_s(Y)\cdot\nabla u_s(Y)G_s(X_0,Y)dY-\int_{\Omega}A(Y)u(Y)\cdot\nabla u(Y)G(X_0,Y)dY}<\epsilon,
  $$
 which justifies Claim 4.

 Combining Claim 4 and \eqref{I2est1}, we obtain that \eqref{identity} holds for Lipschitz domain $\Omega$, $L=-\divg A\nabla$ satisfying \eqref{aBound}--\eqref{bBMO}, and $f\in Lip(\bdy\Omega)$.

\underline{Step 4.} Let $f\in L^2(\bdy\Omega,d\omega)$, $u(X)=\int_{\bdy\Omega}fd\omega^X$. Let $f_i\rightarrow f$ in $L^2(\bdy\Omega,d\omega)$, with each $f_i\in Lip(\bdy\Omega)$. Define $u_i(X)=\int_{\bdy\Omega}f_id\omega^X$.

By Step 3, we have
$$
\int_{\Omega} A\nabla u_i\cdot\nabla u_iG=-\frac1{2}u_i(X_0)^2+\frac1{2}\int_{\bdy\Omega}f_i^2d\omega^{X_0}<\infty,
$$
$\Rightarrow$
\begin{align}\label{Auiui}
  \lim_{i\rightarrow\infty}\int_{\Omega} A\nabla u_i\cdot\nabla u_iG
  &=\lim_{i\rightarrow\infty}
  \Big(-\frac1{2}u_i(X_0)^2+\frac1{2}\int_{\bdy\Omega}f_i^2d\omega^{X_0}\Big)\nonumber\\
  &=-\frac1{2}u(X_0)^2+\frac1{2}\int_{\bdy\Omega}f^2d\omega^{X_0}<\infty
\end{align}

By Caccioppoli's inequality, $\nabla u_i$ converges to $\nabla u$ in $L^2(K)$ where $K$ is any compact subset of $\Omega$, and thus there exists a subsequence, still denoted by $\{\nabla u_i\}$, converges to $\nabla u$ a.e. in $K$.

Therefore, Fatou's lemma implies that
\begin{align}\label{auuGFinite}
  \int_{\Omega}a\nabla u\cdot\nabla uG&=\int_{\Omega}A\nabla u\cdot\nabla uG
  \le\sup_K\lim_{i\rightarrow\infty}\int_K A\nabla u_i\cdot\nabla u_iG\nonumber\\
  &\le \lim_{i\rightarrow\infty}\int_{\Omega} A\nabla u_i\cdot\nabla u_iG
  =-\frac1{2}u(X_0)^2+\frac1{2}\int_{\bdy\Omega}f^2d\omega^{X_0}<\infty.
\end{align}

Similarly, for any fixed $i$, we have
\begin{align*}
  \int_{\Omega}a\nabla (u-u_i)\cdot\nabla (u-u_i)G&=\int_{\Omega}A\nabla (u-u_i)\cdot\nabla (u-u_i)G\\
  &\le\sup_K\lim_{j\rightarrow\infty}\int_K A\nabla(u_j-u_i)\cdot\nabla (u_j-u_i)G\\
  &\le-\frac1{2}\Big(u(X_0)-u_i(X_0)\Big)^2+\frac1{2}\int_{\bdy\Omega}(f-f_i)^2d\omega^{X_0}.
\end{align*}
So for any $\epsilon>0$, there exists $i$ sufficiently large, such that
\begin{equation}\label{au-uiSmall}
  \int_{\Omega}a\nabla (u-u_i)\cdot\nabla (u-u_i)G=\int_{\Omega}A\nabla (u-u_i)\cdot\nabla (u-u_i)G<\epsilon.
\end{equation}

We write
\begin{equation}\label{diffauuG}
  \abs{\int_{\Omega}a\nabla u\cdot\nabla uG-\int_{\Omega}a\nabla u_i\cdot\nabla u_iG}\le
  \abs{\int_{\Omega}a\nabla(u-u_i)\cdot\nabla uG}+\abs{\int_{\Omega}a\nabla(u-u_i)\cdot\nabla u_iG}.
\end{equation}
For the first integral on the right-hand side of \eqref{diffauuG}, we have
\begin{align*}
  \abs{\int_{\Omega}a\nabla(u-u_i)\cdot\nabla uG}
  &\le\frac{\Lambda}{\lambda}
  \Big(\int_{\Omega}a\nabla(u-u_i)\cdot\nabla(u-u_i)G\Big)^{1/2}
  \Big(\int_{\Omega}a\nabla u\cdot\nabla uG\Big)^{1/2}\\
  &<\frac{\Lambda}{\lambda}\Big(\int_{\Omega}a\nabla u\cdot\nabla uG\Big)^{1/2}\epsilon.
\end{align*}
The second integral can be estimated similarly, and is bounded by
$$
\frac{\Lambda}{\lambda}\Big(\int_{\Omega}a\nabla u_i\cdot\nabla u_iG\Big)^{1/2}\epsilon
$$

So by \eqref{Auiui} and \eqref{auuGFinite}, we have
$$
\int_{\Omega}A \nabla u\cdot\nabla uG=\lim_{i\rightarrow\infty}\int_{\Omega}A \nabla u_i\cdot\nabla u_iG=-\frac1{2}u(X_0)^2+\frac1{2}\int_{\bdy\Omega}f^2d\omega^{X_0},
$$
which completes the proof.

\end{proof}

We end this section by pointing out that the $\epsilon-$approximability result in \cite{kenig2000new} and the result in \cite{kenig2016square} also hold for operators defined in Section 2 and domains being Lipschitz. Namely, we have the following

\begin{thm}\label{KKPT00}
  Let $L=\divg A\nabla$ be as in Section 2. Let $\Omega\subset\Rn$ be a Lipschitz domain, $X_0\in\Omega$. Then there exists an $\epsilon$, depending on $\lambda,\Lambda,\Gamma$ and the Lipschitz character of $\Omega$ such that if every solution $u$ to $Lu=0$, with $\norm{u}_{L^{\infty}}\le 1$, is $\epsilon-$approxiable on $\Omega$, then $\omega^{X_0}_L\in A_{\infty}(d\sigma)$, where $d\sigma$ is the surface measure on $\bdy\Omega$. 
\end{thm}

The proof proceeds exactly as in \cite{kenig2000new} since the solutions have all the properties to prove the theorem: the Harnack principle, H\"older continuity and boundary H\"older continuity.

Similarly, we obtain the following results in \cite{kenig2016square}:
\begin{thm}
  Let $L$, $\Omega$ and $X_0$ be as in Theorem \ref{KKPT00}. Assume that there exists some $A<\infty$ such that for all Borel sets $H\subset\bdy\Omega$, the solution to the Dirichlet problem
  $$
  \begin{cases}
  Lu=0 \quad\text{in }\Omega\\
  u=\chi_H \quad\text{on }\bdy\Omega.
  \end{cases}
  $$
  satisfies the following Carleson bound
  $$
  \sup_{\Delta\subset\bdy\Omega,\diam(\Delta)\le\diam(\Omega)}\frac1{\sigma(\Delta)}\int_{T(\Delta)}\delta(X)\abs{\nabla u(X)}^2dX\le A.
  $$
 Then $\omega^{X_0}_L\in A_{\infty}(d\sigma)$.
\end{thm}

\begin{thm}
Let $L$, $\Omega$ and $X_0$ be as in Theorem \ref{KKPT00}. Assume that for all Borel sets $H\subset\bdy\Omega$, the solution to the Dirichlet problem
  $$
  \begin{cases}
  Lu=0 \quad\text{in }\Omega\\
  u=\chi_H \quad\text{on }\bdy\Omega.
  \end{cases}
  $$
  satisfies 
  \begin{equation}
      \norm{S_{\alpha}(u)}_{L^p(\bdy\Omega,d\sigma)}\le A\norm{u^*}_{L^p(\bdy\Omega,d\sigma)} 
  \end{equation}
  for some $p$, $1+\frac1{n-2}\le p<\infty$ if $n\ge 3$ or $p_0\le p<\infty$ if $n=2$ with $p_0$ depending on $\lambda,\Lambda,\Gamma$ and the Lipschitz character of $\Omega$. Then $\omega^{X_0}_L\in A_{\infty}(d\sigma)$.
\end{thm}

\newpage
\bibliography{project}
\bibliographystyle{amsalpha}

\Addresses

\end{document}